\documentclass[11pt]{amsart}
\usepackage{amsmath,amsthm,amssymb,enumerate,mathscinet,mathtools}
\usepackage{fullpage}
\usepackage{graphicx}
\usepackage{verbatim}
\usepackage{mathrsfs}
\usepackage{hyperref}
\newcommand{\msc}[1]{\begin{center}MSC2000: #1.\end{center}}
\newtheorem{theorem}{Theorem}[section]
\newtheorem{fact}[theorem]{Fact}
\newtheorem{lemma}[theorem]{Lemma}

\newtheorem{proposition}[theorem]{Proposition}
\newtheorem{corollary}[theorem]{Corollary}
\newtheorem{conjecture}[theorem]{Conjecture}

\newtheorem{definition}[theorem]{Definition}

\newtheorem{thm}[theorem]{Theorem}
\newtheorem{prop}[theorem]{Proposition}

\newtheorem{conj}[theorem]{Conjecture}

\numberwithin{equation}{section}

\def\eps{\varepsilon}

\def\LL{\mathcal{L}}
\def\res{\mathrm{res}}
\def\ex{\mathrm{ex}}
\def\Hom{\mathrm{Hom}}
\DeclareMathOperator{\rank}{rank}

\DeclarePairedDelimiter\floor{\lfloor}{\rfloor} 
\DeclarePairedDelimiter\ceil{\lceil}{\rceil}  

\def\COMMENT#1{}

\allowdisplaybreaks

\title{Independent sets in hypergraphs and Ramsey properties of graphs and the integers}

\author{Robert Hancock, Katherine Staden and
 Andrew Treglown}

\thanks{RH: University of Birmingham, United Kingdom, and Czech Academy of Sciences, Prague, Czechia, {\tt hancock@math.cas.cz}. KS: University of Oxford, United Kingdom, {\tt staden@maths.ox.ac.uk}. Research supported by ERC Grant 306493.
AT: University of Birmingham, United Kingdom, {\tt a.c.treglown@bham.ac.uk}. Research supported by EPSRC grant EP/M016641/1.}

\begin{document}

\begin{abstract}
Many important problems in combinatorics and other related areas can be phrased in the language of independent sets in hypergraphs. Recently Balogh, Morris and Samotij~\cite{container1}, and independently Saxton and Thomason~\cite{container2} developed very general container theorems for independent sets in hypergraphs; both of which have seen numerous applications to a wide range of problems.
In this paper we use the container method to give relatively short and elementary proofs of a number of results concerning Ramsey (and Tur\'an properties) of (hyper)graphs and the integers.
In particular:
\begin{itemize}

\item We generalise the \emph{random Ramsey theorem} of R\"odl and Ruci\'nski~\cite{random1, random2, random3} by providing a resilience analogue. Our result 
unifies and generalises several fundamental results in the area including
 the \emph{random version of Tur\'an's theorem} due to Conlon and Gowers~\cite{conlongowers} and Schacht~\cite{schacht}.
\item The above result also resolves a general subcase of the \emph{asymmetric random Ramsey conjecture} of Kohayakawa and Kreuter~\cite{kreu}.
\item All of the above results in fact hold for uniform hypergraphs.
\item For a (hyper)graph $H$, we  determine, up to an error term in the exponent, the number of $n$-vertex (hyper)graphs $G$ that have the Ramsey property with respect to $H$ (that is, whenever $G$ is $r$-coloured, there is a monochromatic copy of $H$ in $G$).
\item We strengthen the \emph{random Rado theorem} of Friedgut, R\"odl and Schacht~\cite{frs} by proving a resilience version of the result.
\item For partition regular matrices $A$ we  determine, up to an error term in the exponent, the number of subsets of $\{1,\dots,n\}$ for which there exists an $r$-colouring which contains no monochromatic solutions to $Ax=0$.
\end{itemize}
Along the way a number of open problems are posed.

\end{abstract}

\date{\today}

\maketitle
\msc{5C30, 5C55, 5D10, 11B75}

\section{Introduction}

Recently, the
 \emph{container method} has  developed  as a powerful tool for attacking problems which reduce to counting independent sets in (hyper)graphs. Loosely speaking, container results typically state that the independent sets of a given (hyper)graph $H$ lie only in a `small' number of subsets of the vertex set of $H$ (referred to as \emph{containers}), where each of these containers is an `almost independent set'. 
The method has been of particular importance because a diverse range of problems in combinatorics and other areas can be  rephrased into this setting. For example,  container results have been used to tackle problems arising  in Ramsey theory, combinatorial number theory, positional games, list colourings of graphs and $H$-free graphs.

Although the container method has seen an explosion in applications over the last few years, the technique actually dates back to work of Kleitman and Winston~\cite{kw1, kw2} from more than 30 years ago; they constructed a relatively simple algorithm that can be used to produce \emph{graph} container results. The catalysts for recent advances in the area are the \emph{hypergraph} container theorems of Balogh, Morris and Samotij~\cite{container1} and  Saxton and Thomason~\cite{container2}. Both works yield very general container theorems for hypergraphs whose edge distribution satisfies  certain boundedness conditions. These results are also  related to general transference theorems of Conlon and Gowers~\cite{conlongowers} and Schacht~\cite{schacht}. In particular, these container  and transference theorems can be used to prove a range of combinatorial results in a random setting. See~\cite{bms2} for a survey on the container method.

An overarching aim of the paper is to demonstrate that with the container method at hand, one can give relatively short and elementary proofs of fundamental results concerning Ramsey properties of graphs and the integers. Moreover, our results 
give us a precise understanding about how \emph{resiliently} typical graphs and sets of integers of a given density possess a given Ramsey property.
 In particular, one of our main results is a resilience random Ramsey theorem (Theorem~\ref{ramres}). This result provides a unified framework for studying both the Ramsey and Tur\'an problems in the setting of random (hyper)graphs. In particular, Theorem~\ref{ramres} implies the (so-called 1-statements of the) random Ramsey theorem due to R\"odl and Ruci\'nski~\cite{random1, random2, random3} and the random version of Tur\'an's theorem~\cite{conlongowers, schacht}. Moreover, Theorem~\ref{ramres} also resolves a general subcase  of the asymmetric random Ramsey conjecture of Kohayakawa and Kreuter~\cite{kreu}.
Since Theorem~\ref{ramres} unifies and generalises several fundamental results concerning Ramsey and Tur\'an properties of random (hyper)graphs, we survey these topics in Sections~\ref{supersub}--\ref{supersub2} before we state this result in Section~\ref{supersub3}.

We also prove a sister result to Theorem~\ref{ramres}, a resilience strengthening of the random Rado theorem (Theorem~\ref{radores1new}). Again the container method allows us to give a rather short proof of this result. We further provide results on the enumeration of Ramsey graphs (Theorem~\ref{ramseycount}) and sets of integers without a given Ramsey property (Theorem~\ref{radonumnew}).

The results we prove all correspond to problems concerning tuples of disjoint independent sets in hypergraphs. 
In particular, from the container theorem of Balogh, Morris and Samotij one can easily obtain an analogous result for tuples of  independent sets in hypergraphs (see Proposition~\ref{multi}).
It turns out that many Ramsey-type questions (and other problems) can be naturally phrased in this setting. For example, by Schur's theorem we know that, if $n$ is large, then whenever one $r$-colours the elements of
$[n]:=\{1,\dots,n\}$ there is a monochromatic solution to $x+y=z$. This raises the question of how large can a subset $S \subseteq [n]$ be whilst failing to have this property?
(This problem was first posed back in 1977 by Abbott and Wang~\cite{aw}.)
 Let $H$ be the hypergraph with vertex set $[n]$ in which edges precisely correspond to solutions to $x+y=z$. (Note $H$ will have edges of size $2$ and $3$.)
Then sets $S \subseteq [n]$ without this property are precisely the union of $r$ disjoint independent sets in $H$.

In Section~\ref{sec3} we state the container theorem for tuples of independent sets in hypergraphs.
In Sections~\ref{sec4} and~\ref{sec5} we give our  applications of this container result to enumeration and resilience questions arising in Ramsey theory for graphs and the integers.

\subsection{Resilience in hypergraphs and the integers}\label{ressec}

\subsubsection{Resilience in graphs}
The notion of \emph{graph resilience} has received significant attention in recent years. Roughly speaking, resilience concerns the question of how `strongly'  a graph $G$ satisfies a certain monotone graph property $\mathcal P$. 
\emph{Global resilience} concerns how many edges  one can delete and still ensure the resulting graph has property $\mathcal P$ whilst \emph{local resilience} considers how many edges one can delete at each vertex whilst ensuring the resulting graph has property $\mathcal P$. More precisely, we define the \emph{global resilience of $G$ with respect to $\mathcal{P}$}, $\res(G,\mathcal{P})$, to be the minimum number $t$ such that by deleting $t$ edges from $G$, one can obtain a graph not having $\mathcal{P}$. Many classical results in extremal combinatorics can be rephrased in terms of resilience. For example, Tur\'an's theorem determines the global resilience of $K_n$ with respect to the property of containing $K_r$ (where $r<n$) as a subgraph. 

The systematic study of graph resilience was initiated in a paper of Sudakov and Vu~\cite{suvu}, though such questions had been studied before this.
In particular, a key question in the area is to establish the resilience of various properties of the Erd\H{o}s--R\'enyi random graph $G_{n,p}$.
(Recall that  $G_{n,p}$ has vertex set $[n]$ in which each possible edge is present with probability $p$, independent of all other  choices.)
The local resilience of $G_{n,p}$ has been investigated, for example, with respect to Hamiltonicity~ e.g.~\cite{suvu, lee}, almost spanning trees~\cite{bcs} and embedding subgraphs of small bandwidth~\cite{bkt}.
See~\cite{suvu} and the surveys~\cite{conlonsurvey, sudakovsurvey} for further background on the subject. 
In this paper we study the global resilience of $G_{n,p}$ with respect to Ramsey properties (in fact, as we explain later, we will consider its hypergraph analogue $G^{(k)}_{n,p}$ for $k \geq 2$). 
First we will focus on the graph case.

\subsubsection{Ramsey properties of random graphs}\label{supersub}
An event occurs in $G_{n,p}$ \emph{with high probability} (w.h.p.) if its probability tends to $1$ as $n \rightarrow \infty$.
For many properties $\mathcal{P}$ of $G_{n,p}$, the probability that $G_{n,p}$ has the property exhibits a phase transition, changing from $0$ to $1$ over a small interval.
That is, there is a \emph{threshold} for $\mathcal{P}$: a function $p_0 = p_0(n)$ such that $G_{n,p} \text{ has }\mathcal{P}$ w.h.p.~when $p \gg p_0$ (the \emph{$1$-statement}), while $G_{n,p} \text{ does not have }\mathcal{P}$ w.h.p.~when $p \ll p_0$ (the \emph{$0$-statement}).
Indeed, Bollob\'as and Thomason~\cite{bolthom} proved that every \emph{monotone} property $\mathcal{P}$ has a threshold.

Given a graph $H$, set $d_2(H):=0$ if $e(H)=0$; $d_2(H):=1/2$ when $H$ is precisely an edge and define $d_2(H) := (e(H)-1)/(v(H)-2)$ otherwise.
Then define $m_2(H) := \max_{H' \subseteq H}d_2(H')$ to be the \emph{$2$-density} of $H$.
This graph parameter turns out to be very important when determining the threshold for certain properties in $G_{n,p}$ concerning the containment of a small subgraph $H$, which we explain further below.

Given $\eps >0$ and a graph $H$, we say that a graph $G$ is \emph{$(H,\eps)$-Tur\'an} if every subgraph of $G$ with at least
$(1-\frac{1}{\chi (H)-1}+\eps)e(G)$ edges contains a copy of $H$.
Note that the Erd\H{o}s--Stone theorem implies that $K_n$ is $(H,\eps)$-Tur\'an for any fixed $H$ provided $n$ is sufficiently large.
To motivate the definition, consider any graph $G$. Then by considering a random 
  partition of $V(G)$ into $\chi(H)-1$ parts (and then removing any edge contained within a part)
we see that there is a subgraph $G'$ of $G$ that is $(\chi(H)-1)$-partite where $e(G')\geq (1-\frac{1}{\chi(H)-1})e(G)$. In particular, $H\not \subseteq G'$.
Intuitively speaking, this implies that (up to the $\eps$ term), $(H,\eps)$-Tur\'an graphs are those graphs that most \emph{strongly} contain $H$.

Rephrasing to  the language of resilience, we see that if, for any $\eps>0$, $G$ is $(H,\eps)$-Tur\'an, then $\res(G,\mathcal P) = (\frac{1}{\chi(H)-1} \pm \eps)e(G)$, and vice versa, where $\mathcal P$ is the property of containing $H$ as a subgraph.
(Note that we write $x = a\pm b$ to say that the value of $x$ is some real number in the interval $[a-b,a+b]$.)
The global resilience of $G_{n,p}$ with respect to the Tur\'an problem has been extensively studied. 
Indeed, a recent trend in combinatorics and probability concerns so-called \emph{sparse random} analogues of extremal theorems (see~\cite{conlonsurvey}), and determining when $G_{n,p}$ is $(H,\eps)$-Tur\'an is an example of such a result.

If $p \leq cn^{-1/m_2(H)}$ for some small constant $c$,
then it is not hard to show that w.h.p.~$G_{n,p}$ is not $(H,\eps)$-Tur\'an.
In \cite{hax1, hax2, klr} it was conjectured that w.h.p.~$G_{n,p}$ is $(H,\eps)$-Tur\'an provided that $p\geq Cn^{-1/m_2(H)}$, where $C$ is a (large) constant.
After a number of partial results, this conjecture was confirmed by Schacht~\cite{schacht} and (in the case when $H$ is strictly $2$-balanced, i.e. $m_2(H')<m_2(H)$ for all $H' \subset H$) by Conlon and Gowers~\cite{conlongowers} .

\begin{thm}[\cite{schacht, conlongowers}]\label{randomturan}
For any graph $H$ with $\Delta(H) \geq 2$ and any $\eps >0$, there are positive constants $c,C$ such that
$$\lim _{n \rightarrow \infty} \mathbb P [G_{n,p} \text{ is } (H,\eps)\text{-Tur\'an}]=\begin{cases}
0 &\text{ if } p<cn^{-1/m_2(H)}; \\
 1 &\text{ if } p>Cn^{-1/m_2(H)}.\end{cases}$$
\end{thm}

Given an integer $r$, an \emph{$r$-colouring} of a graph $G$ is a function $\sigma:E(G) \rightarrow [r]$. (So this is not necessarily a proper colouring.) 
We say that $G$ is \emph{$(H,r)$-Ramsey} if every $r$-colouring of $G$ yields a monochromatic copy of $H$ in $G$. 
Observe that being $(H,1)$-Ramsey is the same as containing $H$ as a subgraph.
So the 1-statement of Theorem~\ref{randomturan} says that, given $\eps>0$, there exists a positive constant $C$ such that, if $p > Cn^{-1/m_2(H)}$, then
\begin{equation}\label{resintro}
\lim_{n \rightarrow \infty} \mathbb{P}\left[\frac{\res(G_{n,p},(H,1)\text{-Ramsey})}{e(G_{n,p})} = \frac{1}{\chi(H)-1} \pm \eps \right] = 1.
\end{equation}


The following result of R\"odl and Ruci\'nski~\cite{random1, random2, random3} yields a random version of Ramsey's theorem.

\begin{thm}[\cite{random1, random2, random3}]\label{randomramsey} Let $r \geq 2$ \COMMENT{The case $r=1$ (which is simply $G_{n,p}$ containing $H$) has threshold $n^{-1/m(H)}$, proven for balanced graphs by Erd\H{o}s and R\'enyi, 1960.}be a positive integer and let $H$ be a graph that is not a forest consisting of stars and paths of length $3$.
There are positive constants $c,C$ such that
$$\lim _{n \rightarrow \infty} \mathbb P [G_{n,p} \text{ is } (H,r)\text{-Ramsey}]=\begin{cases}
0 &\text{ if } p<cn^{-1/m_2(H)}; \\
 1 &\text{ if } p>Cn^{-1/m_2(H)}.\end{cases}$$
\end{thm} 
Thus $n^{-1/m_2(H)}$ is again the threshold for the $(H,r)$-Ramsey property.
Let us provide some intuition as to why.
The expected number of copies of $H$ in $G_{n,p}$ is $\Theta(n^{v(H)}p^{e(H)})$, while the expected number of edges in $G_{n,p}$ is $\Theta(pn^2)$.
When $p = \Theta(n^{-1/d_2(H)})$, these quantities agree up to a constant.
Suppose that $H$ is \emph{2-balanced},~i.e. $d_2(H)=m_2(H)$.%
\COMMENT{KS: strictly balanced already defined on p3 (We also say that $H$ is \emph{strictly $2$-balanced} if $d_2(H')<d_2(H)$ for all $H' \subsetneq H$.)}
For small $c>0$, when $p<cn^{-1/m_2(H)}$, most copies of $H$ in $G_{n,p}$ contain an edge which appears in no other copy.
Thus we can hope to colour these special edges blue and colour the remaining edges red to eliminate all monochromatic copies of $H$.
For large $C>0$, most edges lie in many copies of $H$, so the copies of $H$ are highly overlapping and we cannot avoid monochromatic copies.
In general, when $H$ is not necessarily 2-balanced, the threshold is $n^{-1/d_2(H')}$ for the `densest' subgraph $H'$ of $H$ since, roughly speaking, the appearance of $H$ is governed by the appearance of its densest part.

We remark that Nenadov and Steger~\cite{ns} recently gave a short proof of Theorem~\ref{randomramsey} using the container method.

\subsubsection{Asymmetric Ramsey properties in random graphs}

It is natural to ask for an \emph{asymmetric} analogue of Theorem~\ref{randomramsey}.
Now, for graphs $H_1,\ldots,H_r$, a graph $G$ is \emph{$(H_1,\ldots,H_r)$-Ramsey} if  for any $r$-colouring  of $G$ there is a copy of $H_i$ in colour $i$ for some   $i \in [r]$.
(This definition coincides with that of $(H,r)$-Ramsey when $H_1=\ldots=H_r=H$.)
Kohayakawa and Kreuter~\cite{kreu} conjectured an analogue of Theorem~\ref{randomramsey} in the asymmetric case.
To state it, we need to introduce the \emph{asymmetric density} of $H_1,H_2$ where $m_2(H_1)\geq m_2(H_2)$ via
\begin{equation}\label{2density}
m_2(H_1,H_2):=\max \left \{\frac{e(H_1')}{v(H_1')-2+1/m_2(H_2)}: H_1' \subseteq H_1 \text{ and } e(H_1') \geq 1 \right \}.
\end{equation}

\begin{conjecture}[\cite{kreu}]\label{conjkreu}
For any graphs $H_1,\ldots,H_r$ with $m_2(H_1) \geq \ldots \geq m_2(H_r)>1$, there are positive constants $c,C > 0$ such that
$$
\lim_{n \rightarrow \infty} \mathbb{P}\left[ G_{n,p} \text{ is } (H_1,\ldots,H_r)\text{-Ramsey} \right] =\begin{cases}
0 &\text{ if } p<cn^{-1/m_2(H_1,H_2)}; \\
 1 &\text{ if } p>Cn^{-1/m_2(H_1,H_2)}.\end{cases}$$
\COMMENT{This is a coarse threshold: They say `sharp' in paper. Also they actually state it with $r=2$.}
\end{conjecture}

So the conjectured threshold only depends on the `joint density' of the densest two graphs $H_1,H_2$.
The intuition for this threshold is discussed in detail e.g. in Section~1.1 in~\cite{asymmramsey}.
One can show that $m_2(H_1) \geq m_2(H_1,H_2) \geq m_2(H_2)$ with equality if and only if $m_2(H_1)=m_2(H_2)$.
Thus Conjecture~\ref{conjkreu} would generalise Theorem~\ref{randomramsey}.
Kohayakawa and Kreuter~\cite{kreu} have confirmed Conjecture~\ref{conjkreu} when the $H_i$ are cycles. 
In~\cite{msss} it was observed that the approach used by Kohayakawa and Kreuter~\cite{kreu}  implies the 1-statement of Conjecture~\ref{conjkreu} holds when $H_1$ is strictly 2-balanced \emph{provided} the so-called K\L R conjecture holds.
This latter conjecture was proven by Balogh, Morris and Samotij~\cite{container1} thereby proving the 1-statement of Conjecture~\ref{conjkreu} holds in this case.

{\bf Additional note:} Since the paper was submitted the 1-statement of Conjecture~\ref{conjkreu} has been proven by Mousset, Nenadov and Samotij~\cite{mns}.

\subsubsection{Ramsey properties of random hypergraphs}\label{supersub2}
Consider now the $k$-uniform analogue $G^{(k)}_{n,p}$ of $G_{n,p}$ which has vertex set $[n]$ and in which every $k$-element subset of $[n]$ appears as an edge with probability $p$, independent of all other choices.
Here, we wish to obtain analogues of Theorems~\ref{randomturan},~\ref{randomramsey} and Conjecture~\ref{conjkreu} by determining the threshold for being $(H,\eps)$-Tur\'an, $(H,r)$-Ramsey, and more generally being $(H_1,\ldots,H_r)$-Ramsey.
The definitions of $(H,r)$-Ramsey and $(H_1,\ldots,H_r)$-Ramsey extend from graphs in the obvious way.
Given a $k$-uniform hypergraph $H$, let $\ex(n;H)$ be the maximum size of an $n$-vertex $H$-free hypergraph.
A simple averaging argument shows that the limit
$$
\pi(H) := \lim_{n \rightarrow \infty} \frac{\ex(n;H)}{\binom{n}{k}}
$$
exists.
Now we say that a $k$-uniform hypergraph $G$ is \emph{$(H,\eps)$-Tur\'an} if every subhypergraph of $G$ with at least $(\pi(H)+\eps)e(G)$ edges contains a copy of $H$.
(Since $\pi(H)=1-\frac{1}{\chi(H)-1}$ when $k=2$, this generalises the definition we gave earlier.)
We also need to generalise the notion of $2$-density to \emph{$k$-density}:
Given a $k$-graph $H$, define
$$
d_k(H) := \begin{cases}
0 &\text{ if } e(H)=0; \\
1/k &\text{ if } v(H)=k \text{ and } e(H)=1;\\
\frac{e(H)-1}{v(H)-k} &\text{ otherwise,}
\end{cases}
$$
and let
$$m_k(H) := \max_{H' \subseteq H}d_k(H').$$
The techniques of Conlon--Gowers~\cite{conlongowers} and of Schacht~\cite{schacht} actually extended to a proof of a version of Theorem~\ref{randomturan} for hypergraphs:

\begin{thm}[\cite{conlongowers,schacht}]\label{hyprandomturan}
For any $k$-uniform hypergraph $H$ with maximum vertex degree at least two and any $\eps >0$, there are positive constants $c,C$ such that
$$\lim _{n \rightarrow \infty} \mathbb P [G^{(k)}_{n,p} \text{ is } (H,\eps)\text{-Tur\'an}]=\begin{cases}
0 &\text{ if } p<cn^{-1/m_k(H)}; \\
 1 &\text{ if } p>Cn^{-1/m_k(H)}.\end{cases}$$
\end{thm}

The $1$-statement of Theorem~\ref{randomramsey} was generalised to hypergraphs by Friedgut, R\"odl and Schacht~\cite{frs} and by Conlon and Gowers~\cite{conlongowers}, proving a conjecture of R\"odl and Ruci\'nski~\cite{random5}. (The special cases of the complete $3$-uniform hypergraph $K^{(3)}_4$ on four vertices and of $k$-partite $k$-uniform hypergraphs were already proved in~\cite{random5},~\cite{rrs} respectively. Also in~\cite{ns} Nenadov and Steger remark that their proof of the 1-statement of Theorem~\ref{randomramsey} extends to Theorem~\ref{hyprandomramsey}.)

\begin{theorem}[\cite{conlongowers,frs}]\label{hyprandomramsey}
Let $r,k \geq 2$ be integers and let $H$ be a $k$-uniform hypergraph with maximum vertex degree at least two.
There is a positive constant $C$ such that
$$
\lim_{n \rightarrow \infty} \mathbb{P}[G^{(k)}_{n,p} \text{ is } (H,r)\text{-Ramsey}] = 1 \quad\text{ if } p > Cn^{-1/m_k(H)}.
$$
\end{theorem}

In~\cite{asymmramsey}, sufficient conditions are given for a corresponding $0$-statement. However, the authors further show that, for $k \geq 4$, there is a $k$-uniform hypergraph $H$ such that the threshold for $G^{(k)}_{n,p}$ to be $(H,r)$-Ramsey is \emph{not} $n^{-1/m_k(H)}$, and \emph{nor} does it correspond to the exceptional case in the graph setting of certain forests, where there is a coarse threshold due to the appearance of small subgraphs. (This $H$ is the disjoint union of a tight cycle and hypergraph triangle.)

For the asymmetric Ramsey problem, we need to suitably generalise~(\ref{2density}), in the obvious way:
for any $k$-uniform hypergraphs $H_1,H_2$ with non-empty edge sets and $m_k(H_1) \geq m_k(H_2)$, let
\begin{equation}\label{kdensity}
m_k(H_1,H_2):=\max \left \{\frac{e(H_1')}{v(H_1')-k+1/m_k(H_2)}: H_1' \subseteq H_1 \text{ and } e(H_1') \geq 1 \right \}
\end{equation}
be the \emph{asymmetric $k$-density} of $(H_1,H_2)$.
Again,
$$
m_k(H_1) \geq m_k(H_1,H_2) \geq m_k(H_2), 
$$
so, in particular, $m_k(H_1,H_2)=m_k(H_1)$ if and only if $H_1$ and $H_2$ have the same $k$-density.

Recently, Gugelmann, Nenadov, Person, Steger, {\v S}kori{\'c} and Thomas~\cite{asymmramsey} generalised the $1$-statement of Conjecture~\ref{conjkreu} to $k$-uniform hypergraphs, in the case when $H_1'=H_1$ is the unique maximiser in~(\ref{kdensity}), i.e. $H_1$ is \emph{strictly $k$-balanced with respect to $m_k(\cdot,H_2)$}.

\begin{theorem}[\cite{asymmramsey}]\label{asymmthm}
For all positive integers $r,k$ with $k \geq 2$ and $k$-uniform hypergraphs $H_1,\ldots,H_r$ with $m_k(H_1) \geq \ldots \geq m_k(H_r)$ where $H_1$ is strictly $k$-balanced with respect to $m_k(\cdot,H_2)$, there exists $C > 0$ such that
$$
\lim_{n \rightarrow \infty} \mathbb{P}\left[ G^{(k)}_{n,p} \text{ is } (H_1,\ldots,H_r)\text{-Ramsey} \right] = 1 \quad\text{ if } p > Cn^{-1/m_k(H_1,H_2)}.
$$
\end{theorem}

They further prove a version of Theorem~\ref{asymmthm} with the weaker bound $p > Cn^{-1/m_k(H_1,H_2)}\log n$ when $H_1$ is not required to be strictly $k$-balanced with respect to $m_k(\cdot,H_2)$.

\subsubsection{New resilience result}\label{supersub3}

Our main result here is Theorem~\ref{ramres}, which generalises, fully and partially, all of the $1$-statements of the results discussed in this section, giving a unified setting for both the random Ramsey theorem and the random Tur\'an theorem.
Once we have obtained a container theorem for Ramsey graphs (Theorem~\ref{ramseycont}), the proof is short (see Section~\ref{randomsec}).

For $k$-uniform hypergraphs $H_1,\ldots,H_r$ and a positive integer $n$, let $\ex^r(n;H_1,\ldots,H_r)$ be the maximum size of an $n$-vertex $k$-uniform hypergraph $G$ which is not $(H_1,\ldots,H_r)$-Ramsey.
Define the \emph{$r$-coloured Tur\'an density}
\begin{equation}\label{piintro}
\pi(H_1,\ldots,H_r) := \lim_{n \rightarrow \infty}\frac{\ex^r(n;H_1,\ldots,H_r)}{\binom{n}{k}}.
\end{equation}
Observe that $\ex^1(n;H) = \ex(n;H)$  since a hypergraph is $H$-free if and only if it is not $(H,1)$-Ramsey.
Note further that $\pi(\cdot,\ldots,\cdot)$ generalises $\pi(\cdot)$. So when $k=2$, we have $\pi(H)=1-\frac{1}{\chi(H)-1}$. 
We will observe in Section~\ref{maxsec} that the limit in~(\ref{piintro}) does indeed exist, so $\pi(\cdot,\ldots,\cdot)$ is well-defined.
Further, crucially for $k$-uniform hypergraphs $H_1,\ldots,H_r$, there exists an $\eps=\eps(H_1,\dots,H_r)>0$ so that $\pi(H_1,\dots,H_r)<1-\eps$ (see (\ref{pibound}) in Section~\ref{maxsec}).

\begin{theorem}[Resilience for random Ramsey]\label{ramres}
Let $\delta>0$, let $r,k$ be positive integers with $k \geq 2$ and let $H_1,\ldots,H_r$ be $k$-uniform hypergraphs each with maximum vertex degree at least two, and such that $m_k(H_1) \geq \ldots \geq m_k(H_r)$. There exists $C > 0$ such that
$$
\lim_{n \rightarrow \infty} \mathbb{P}\left[ \frac{\res\left(G^{(k)}_{n,p} ,(H_1,\ldots,H_r)\text{-Ramsey}\right)}{e\left(G^{(k)}_{n,p}\right)} = 1-\pi(H_1,\ldots,H_r) \pm \delta\right] = 1\quad\text{if } p > Cn^{-1/m_k(H_1)}.
$$
\end{theorem}

Thus, when $p > Cn^{-1/m_k(H_1)}$, the random hypergraph $G^{(k)}_{n,p}$ is w.h.p such that \emph{every} subhypergraph $G'$ with at least a $\pi(H_1,\ldots,H_r)+\Omega(1)$ fraction of the edges is $(H_1,\ldots,H_r)$-Ramsey.
Conversely, there is a subgraph of $G^{(k)}_{n,p}$ whose edge density is slightly smaller than this which \emph{does not} have the Ramsey property.

Note that the threshold of $p>Cn^{-1/m_k(H_1)}$ in Theorem~\ref{ramres} is tight up to the multiplicative constant $C$. Indeed, consider the random hypergraph $G^{(k)}_{n,p}$ with $p \ll n^{-1/m_k(H_1)}$. Let $H_1' \subseteq H_1$ be such that $m_k(H_1)=d_k(H_1')$. Then the expected number of copies of $H_1'$ in $G^{(k)}_{n,p}$ is much smaller than the expected number of edges in $G^{(k)}_{n,p}$, so w.h.p. we can delete every copy of $H_1'$ (and therefore $H_1$) by removing $o(e(G^{(k)}_{n,p}))$ edges. So the hypergraph $G$ that remains has $(1-o(1))e(G^{(k)}_{n,p})$ edges, and is not $(H_1,\ldots,H_r)$-Ramsey because we can colour every edge of $G$ with colour 1. Then, since $G$ is $H_1$-free, there is no  copy of $H_i$ in colour $i$ in $G$.

Let us describe the importance of Theorem~\ref{ramres} (in the case $k=2$ and $H_1=\ldots=H_r=H$) in conjunction with Theorem~\ref{randomramsey}.
The $0$-statement of Theorem~\ref{randomramsey} says that a typical sparse graph, i.e.~one with density at most $cn^{2-1/m_2(H)}$, is~\emph{not} $(H,r)$-Ramsey.
On the other hand, by Theorem~\ref{ramres}, a typical dense graph, i.e.~one with density at least $Cn^{2-1/m_2(H)}$, has the Ramsey property in a sense which is \emph{as strong as possible} with respect to subgraphs: every sufficiently dense subgraph is $(H,r)$-Ramsey, and this minimum density is the largest we could hope to require.

The relationship between Theorem~\ref{ramres} and the previous results stated in this section can be summarised as follows:


\begin{itemize}
\item The $1$-statement of Theorem~\ref{randomturan} is recovered when $k=2$ and $r=1$. This follows from (\ref{resintro}) and the relation between $\pi(H)$ and $\chi(H)$.
\item In the case $k=2$ and $H_1=\ldots=H_r=H$, we obtain a stronger statement in place of the $1$-statement of Theorem~\ref{randomramsey} as described above. 
\item Theorem~\ref{ramres} proves the $1$-statement of Conjecture~\ref{conjkreu} in the case when $m_2(H_1)=m_2(H_2)$ in the same stronger sense as above.
\item The $1$-statement of Theorem~\ref{hyprandomturan} is recovered when $r=1$.
\item Theorem~\ref{ramres} implies Theorem~\ref{hyprandomramsey}, yielding a resilience version of this result.
\item Theorem~\ref{ramres} implies a version of Theorem~\ref{asymmthm} when $m_k(H_1)=m_k(H_2)$ but now $H_1$ is not required to be strictly $k$-balanced with respect to $m_k(\cdot,H_2)$.
\end{itemize}

Note that even though Theorem~\ref{ramres} implies many of the known results concerning Ramsey properties of random (hyper)graphs, often the resilience random Ramsey problem is different to the random Ramsey problem.
In particular, we have determined the threshold for the former problem, whilst we have seen above examples of (hyper)graphs $H_1,\dots, H_r$ where a lower value of $p$ still ensures that $G^{(k)}_{n,p}$ is w.h.p. 
$(H_1,\ldots,H_r)\text{-Ramsey}$.


\subsubsection{Resilience in the integers}
An important branch of Ramsey theory concerns partition properties of sets of integers. Schur's classical theorem~\cite{schur} states that if $\mathbb N$ is 
$r$-coloured there exists a monochromatic solution to $x+y=z$; later van der Waerden~\cite{vdw} showed that the same hypothesis ensures a monochromatic arithmetic progression of arbitrary length. 
More generally,  Rado's theorem~\cite{rado} characterises all those systems of homogeneous linear equations $\mathcal L$ for which every finite colouring of $\mathbb N$ yields a monochromatic solution to $\mathcal L$.

As in the graph case, there has been interest in proving random analogues of such results from arithmetic Ramsey theory. Before we describe the background of this area we will introduce some notation and definitions. Throughout we will assume that $A$ is an $\ell \times k$ integer matrix where $k \geq \ell$ of full rank $\ell$. We will let $\LL (A)$ denote the associated system of linear equations $A x=0$, noting that for brevity we will simply write $\LL$ if it is clear from the context which matrix $A$ it refers to. Let $S$ be a set of integers. If a vector $x=(x_1,\dots,x_k) \in S^k$ satisfies $Ax=0$ (i.e. it is a solution to $\LL$) and the $x_i$ are distinct we call $x$ a \emph{$k$-distinct solution} to $\LL$ in $S$. 

We call a set $S $ of integers \emph{$(\LL,r)$-free} if there exists an $r$-colouring of $S$ such that it contains no monochromatic $k$-distinct  solution to $\LL$. Otherwise we call $S$ \emph{$(\LL,r)$-Rado}.  
In the case when $r=1$, we write $\LL$-free instead of $(\LL,1)$-free.
Define $\mu(n,\LL,r)$ to be the size of the largest $(\LL,r)$-free subset of $[n]$.

 A matrix $A$ is \emph{partition regular} if for any finite colouring of $\mathbb{N}$, there is always a monochromatic solution to $\LL$. 
As mentioned above, Rado's theorem characterises all those integer matrices $A$ that are partition regular.
A matrix $A$ is \emph{irredundant} if there exists a $k$-distinct solution to $\LL$ in $\mathbb{N}$. Otherwise $A$ is \emph{redundant}. 
The study of random versions of Rado's theorem has focused on irredundant partition regular matrices. This is natural since
for every redundant $\ell \times k$ matrix $A$
there exists an irredundant $\ell '\times k'$ matrix $A'$
 for some $\ell '< \ell$ and $k' < k$ with the
same family of solutions (viewed as sets). See~\cite[Section 1]{random4} for a full explanation. 

Another class of matrices that have received attention in relation to this problem are so-called density regular matrices:
An irredundant, partition regular matrix $A$ is \emph{density regular} if any subset $F \subseteq \mathbb{N}$ with positive upper density, i.e., $$ \limsup_{n \to \infty} \frac{|F \cap [n]|}{n} >0,$$ contains a $k$-distinct solution to $\LL$.


Index the columns of $A$ by $[k]$. For a partition $W \dot{\cup} \overline{W} = [k]$ of the columns of $A$, we denote by $A_{\overline{W}}$ the matrix obtained from $A$ by restricting to the columns indexed by $\overline{W}$. Let $\rank(A_{\overline{W}})$ be the rank of $A_{\overline{W}}$, where $\rank(A_{\overline{W}})=0$ for $\overline{W}=\emptyset$. We set 
\begin{align}\label{m(A)def}
m(A):=\max_{\substack{W \dot{\cup} \overline{W} = [k] \\ |W|\geq 2}} \frac{|W|-1}{|W|-1+\rank(A_{\overline{W}})-\rank(A)}.
\end{align} 
We remark that the denominator of $m(A)$ is strictly positive provided that $A$ is irredundant and partition regular.


We now describe some random analogues of results from arithmetic Ramsey theory. Recall that $[n]_p$ denotes a set where each element $a \in [n]$ is included with probability $p$ independently of all other elements. R\"odl and Ruci\'nski~\cite{random4} showed that for irredundant partition regular matrices $A$, $m(A)$ is an important parameter for determining whether $[n]_p$ is $(\mathcal L,r)$-Rado or $(\mathcal L,r)$-free.

\begin{thm}[\cite{random4}]\label{radores0} 
For all irredundant partition regular full rank matrices $A$ and all positive integers $r\geq 2$, there exists a constant $c>0$ such that $$
\lim_{n \rightarrow \infty} \mathbb{P}\left[ [n]_p \text{ is } (\mathcal L,r)\text{-Rado} \right]=0 \quad\text{ if } p < cn^{-1/m(A)}.
$$
\end{thm}
We remark it is important that $r\geq 2$ in Theorem~\ref{radores0}. That is, the corresponding statement for $r=1$ is not true in general.
\COMMENT{In particular if $p=n^{-(\ell-k)/k}$, then the expected number of solutions to $\LL$ in $[n]_p$ is $\Omega(n^{k-\ell} p^k)=\Omega(1)$, and so $\mathbb{P}($There are no solutions to $\LL$ in $[n]_p) \not \to 0$ as $n \to \infty$.}
Roughly speaking, Theorem~\ref{radores0} implies that almost all subsets of $[n]$ with significantly fewer than $n^{1-1/m(A)}$ elements are $(\mathcal L,r)$-free
for any irredundant partition regular matrix $A$.
The following theorem  of Friedgut, R\"odl and Schacht~\cite{frs} complements this result, implying that 
almost all subsets of $[n]$ with significantly more than $n^{1-1/m(A)}$ elements are $(\mathcal L,r)$-Rado
for any irredundant partition regular matrix $A$.

\begin{thm}[\cite{frs}]\label{r3} 
For all irredundant partition regular full rank matrices $A$ and all positive integers $r$, there exists a constant $C>0$ such that $$
\lim_{n \rightarrow \infty} \mathbb{P}\left[ [n]_p \text{ is } (\mathcal L,r)\text{-Rado}\right]=1 \quad\text{ if } p > Cn^{-1/m(A)}.
$$
\end{thm}
Earlier, Theorem~\ref{r3} was confirmed by Graham, R\"odl and Ruci\'nski~\cite{grr} in the case where $\LL$ is $x+y=z$ and $r=2$, and then by R\"odl and Ruci\'nski~\cite{random4} in the case when $A$ is density regular.

Together Theorems~\ref{radores0} and~\ref{r3} show that the threshold for the property of being $(\LL,r)$-Rado is $p=n^{-1/m(A)}$. In light of this, it is interesting to ask if above this threshold 
the property of being $(\LL,r)$-Rado is resilient to the deletion of a significant number of elements. To be precise, given a set $S$, we define the \emph{resilience of $S$ with respect to $\mathcal{P}$}, $\res(S,\mathcal{P})$, to be the minimum number $t$ such that by deleting $t$ elements from $S$, one can obtain a set not having $\mathcal{P}$. For example, when $\mathcal{P}$ is the property of containing an arithmetic progression of length $k$, then Szemer\'edi's theorem can be phrased in terms of resilience; it states that for all $k \geq 3$ and $\eps>0$, there exists $n_0>0$ such that for all integers $n\geq n_0$, we have $\res([n],\mathcal{P}) \geq (1-\eps)n$.

The following result of Schacht~\cite{schacht} provides a resilience strengthening of Theorem~\ref{r3} in the case of density regular matrices.

\begin{thm}[\cite{schacht}]\label{r1} 
For all irredundant density regular full rank matrices $A$, all positive integers $r$ and all $\eps >0$, there exists a constant $C>0$ such that $$
\lim_{n \rightarrow \infty} \mathbb{P}\left[ \frac{\res([n]_p,(\LL,r)\text{-Rado})}{|[n]_p|} \geq 1-\eps \right] = 1 \quad\text{ if } p > Cn^{-1/m(A)}.
$$ 
\end{thm}
Note that in~\cite{schacht} the result is stated in the $r=1$ case only, but the general result follows immediately from this special case.
	
Our next result gives a resilience strengthening of Theorem~\ref{r3} for all irredundant partition regular matrices. 
\begin{thm}\label{radores1new}
For all irredundant partition regular full rank matrices $A$, all positive integers $r$ and all $\delta >0$, there exists a constant $C>0$ such that
$$
\lim_{n \rightarrow \infty} \mathbb{P}\left[ \frac{\res([n]_p,(\LL,r)\text{-Rado})}{|[n]_p|}= 1-\frac{\mu(n,\LL,r)}{n}\pm \delta   \right] = 1 \quad\text{ if } p > Cn^{-1/m(A)}.
$$ 
\end{thm}
It is well known that for all irredundant partition regular full rank matrices $A$ and all positive integers $r$, there exist $n_0=n_0(A,r),\eta=\eta(A,r)>0,$ such that for all integers $n\geq n_0$, we have $\mu(n,\LL,r) \leq (1-\eta)n$. (This follows  from a supersaturation lemma of Frankl, Graham and R\"odl~\cite[Theorem 1]{fgr}.) Thus, Theorem~\ref{radores1new} does imply Theorem~\ref{r3}.
Further, in the case when $A$ is density regular, \cite[Theorem 2]{fgr} immediately implies  that $\mu(n,\LL,r)=o(n)$ for any fixed $r\in \mathbb N$.
Thus Theorem~\ref{radores1new} implies Theorem~\ref{r1}. Theorem~\ref{radores1new} in the case when $r=1$ and $\LL$ is $x+y=z$ was proved by Schacht~\cite{schacht}. In fact, the method of Schacht can be used to prove the theorem  for $r=1$ and every irredundant partition regular matrix $A$.\COMMENT{ I'm pretty certain that this is true; just supersat and Prop~\ref{matmL} needed.}



Intuitively, the reader can interpret Theorem~\ref{radores1new} as stating that almost all subsets of $[n]$ with significantly more than $n^{1-1/m(A)}$ elements  strongly possess the property of being $(\mathcal L,r)$-Rado
for any irredundant partition regular matrix $A$. The `strength' here depends on the parameter  $\mu(n,\LL,r)$. In light of this it is natural to seek good bounds on $\mu(n,\LL,r)$ (particularly in the cases when $\mu(n,\LL,r)=\Omega(n)$). In general,
not too much is known about this parameter. However, as mentioned earlier, in the case when $A=(1,1,-1)$ (i.e. $\mathcal L$ is $x+y=z$), this is (essentially) a 40-year-old problem of Abbott and Wang~\cite{aw}. 
In Section~\ref{secaw} we give an upper bound on $\mu(n,\LL,r)$ in this case for all $r\in \mathbb N$.

Instead of proving Theorem~\ref{radores1new} directly, in Section~\ref{sec4} we will prove a version of the result that holds for a more general class of matrices $A$, and also deals with the asymmetric case, namely Theorem~\ref{radores1}.

{\bf Additional note.} Just before submitting the paper we were made aware of simultaneous and independent work of Spiegel~\cite{sp}. In~\cite{sp} the  case $r=1$ of Theorem~\ref{radores1} is proven.  Spiegel also used the container method
to give an alternative proof of Theorem~\ref{r3}.


\subsection{Enumeration questions for Ramsey problems}\label{countsec}
A fundamental question in combinatorics is to determine the number of structures with a given property. For example, Erd\H{o}s, Frankl and R\"odl~\cite{efr} proved that the number of $n$-vertex $H$-free graphs is $2^{\binom{n}{2}(1-\frac{1}{r-1}+o(1))}$
for any graph $H$ of chromatic number $r$. Here the lower bound follows by considering all the subgraphs of the $(r-1)$-partite Tur\'an graph. 
There has also been interest in strengthening this result e.g. in the case when $H$ is bipartite; see e.g.~\cite{fms,ms}.
Given any $k,r,n \in \mathbb N$ with $k \geq 2$ and $k$-uniform hypergraphs $H_1,\dots, H_r$, define 
${\mathrm{Ram}}(n;H_1,\ldots,H_r)$ to be the collection of all $k$-uniform hypergraphs on vertex set $[n]$ that are $(H_1,\ldots,H_r)$-Ramsey and $\overline{\mathrm{Ram}}(n;H_1,\ldots,H_r)$ to be all those 
$k$-uniform hypergraphs on $[n]$ that are not $(H_1,\ldots,H_r)$-Ramsey.
A natural question is to determine the size of ${\mathrm{Ram}}(n;H_1,\ldots,H_r)$. Surprisingly, we are unaware of any \emph{explicit} results in this direction for $r\geq 2$.
The next application of the container method fully answers this question up to an error term in the exponent. 
\begin{theorem}\label{ramseycount}
Let $k,r,n \in \mathbb{N}$ with $k \geq 2$ and $H_1,\ldots,H_r$ be $k$-uniform hypergraphs.
Then
$$
|\overline{\mathrm{Ram}}(n;H_1,\ldots,H_r)| = 2^{\ex^r(n,H_1,\ldots,H_r)+o(n^k)} = 2^{\pi(H_1,\ldots,H_r)\binom{n}{k}+o(n^k)}.
$$
\end{theorem}%

Note that in the case when $k=2$ and $r=1$, Theorem~\ref{ramseycount} is precisely the above mentioned result of  Erd\H{o}s, Frankl and R\"odl~\cite{efr}.
In fact, one can also obtain Theorem~\ref{ramseycount} by using the work from~\cite{nrs}, a hypergraph analogue of the result in~\cite{efr}; see Section~\ref{sec54} for a proof of this.
Similar results were obtained also using containers by Falgas-Ravry, O'Connell and Uzzell in~\cite{templates}, and by Terry in~\cite{terry} who reproved a result of Ishigami~\cite{ishigami}.

\smallskip
Our final application of the container method determines, up to an error term in the exponent, the number of $(\LL,r)$-free subsets of $[n]$.

\begin{thm}\label{radonumnew}
Let $A$ be an irredundant partition regular matrix of full rank and let $r \in \mathbb N$ be fixed. There are $2^{\mu(n,\LL,r)+o(n)}$ $(\LL,r)$-free subsets of $[n]$.
\end{thm}
As an illustration, a result of Hu~\cite{hu} implies that $\mu (n,\LL, 2)=4n/5+o(n)$ in the case when $\LL$ is $x+y=z$. Thus, Theorem~\ref{radonumnew} tells us all but $2^{(4/5+o(1))n}$ subsets of $[n]$  are
$(\LL,2)$-Rado in this case.
Related results (in the $1$-colour case) were obtained by Green~\cite{G-R} and Saxton and Thomason~\cite{onlinecontainers}.


\section{Notation}
For a (hyper)graph $H$, we define $V(H)$ and $E(H)$ to be the vertex and edge sets of $H$ respectively, and set $v(H):=|V(H)|$ and $e(H):=|E(H)|$. For a set $A \subseteq V(H)$, we define $H[A]$ to be the induced subgraph of $H$ on the vertex set $A$. For an edge set $X \subseteq E(H)$, we define $H - X$ to be hypergraph with vertex set $V(H)$ and edge set $E(H) \setminus X$. 

For a set $A$ and a positive integer $x$, we define $\binom{A}{x}$ to be the set of all subsets of $A$ of size $x$, and we define $\binom{A}{\leq x}$ to be the set of all subsets of $A$ of size at most $x$. We use $\mathcal{P}(X)$ to denote the powerset of $X$, that is, the set of all subsets of $X$. If $B$ is a family of subsets of $A$, then we define $\overline{B}$ to be the complement family, that is, precisely the subsets of $A$ which are not in $B$. 

Given a hypergraph $\mathcal{H}$, for each $T \subseteq V(\mathcal{H})$, we define $\deg_{\mathcal{H}}(T):=|\{e \in E(\mathcal{H}): T \subseteq e\}|$, and let $\Delta_{\ell}(\mathcal{H}):=\max\{\deg_{\mathcal{H}}(T):T \subseteq V(\mathcal{H})$ and $|T|=\ell \}$.

We write $x = a\pm b$ to say that the value of $x$ is some real number in the interval $[a-b,a+b]$.
We use the convention that the set of natural numbers $\mathbb N$ does not include zero.

We will make use of the following Chernoff inequality (see e.g.  \cite[Theorem 2.1, Corollary 2.3]{Janson&Luczak&Rucinski00}). 
\begin{proposition} \label{chernoff}
Suppose $X$ has binomial distribution and $\lambda \geq 0$.
 Then
$$
\mathbb{P}[X > \mathbb{E}[X]+\lambda] \leq \exp\left(-\frac{\lambda^2}{2(\mathbb{E}[X]+\lambda/3)}\right).
$$
Further, if $0<\eps \leq 3/2$ then 
$$\mathbb P[|X-\mathbb E[X]|\geq \eps \mathbb E[X]]\leq 2\exp \left(-\frac{\eps ^2}{3} \mathbb E[X]\right).$$
\end{proposition}



\section{Container results for disjoint independent sets}\label{sec3}
Let $\mathcal{H}$ be a $k$-uniform hypergraph with vertex set $V$. A family of sets $\mathcal{F} \subseteq \mathcal{P}(V)$ is called \em{increasing }\rm if it is closed under taking supersets; in other words for every $A,B \subseteq V$, if $A \in \mathcal{F}$ and $A \subseteq B$, then $B \in \mathcal{F}$. Suppose $\mathcal{F}$ is an increasing family of subsets of $V$ and let $\eps \in (0,1]$. We say that $\mathcal{H}$ is $(\mathcal{F},\eps)$\em{-dense }\rm if $$e(\mathcal{H}[A]) \geq \eps e(\mathcal{H})$$ for every $A \in \mathcal{F}$. We define $\mathcal{I}(\mathcal{H})$ to be the set of all independent sets in $\mathcal{H}$.

The next result is the general hypergraph container theorem of Balogh, Morris and Samotij~\cite{container1}.
\begin{thm}[\cite{container1}, Theorem~2.2]\label{BMS}
For every $k \in \mathbb{N}$ and all positive $c$ and $\eps$, there exists a positive constant $C$ such that the following holds. Let $\mathcal{H}$ be a $k$-uniform hypergraph and let $\mathcal{F} \subseteq \mathcal{P}(V(\mathcal{H}))$ be an increasing family of sets such that $|A| \geq \eps v(\mathcal{H})$ for all $A \in \mathcal{F}$. Suppose that $\mathcal{H}$ is $(\mathcal{F},\eps)$-dense and $p \in (0,1)$ is such that, for every $\ell \in [k]$, 
$$\Delta_{\ell}(\mathcal{H}) \leq c \cdot p^{\ell-1} \frac{e(\mathcal{H})}{v(\mathcal{H})}.$$ 
Then there exists a family $\mathcal{S} \subseteq \binom{V(\mathcal{H})}{\leq Cp\cdot v(\mathcal{H})}$ and functions $f:\mathcal{S} \to \overline{\mathcal{F}}$ and $g:\mathcal{I}(\mathcal{H}) \to \mathcal{S}$ such that for every $I \in \mathcal{I}(\mathcal{H})$, we have that $g(I) \subseteq I$ and $I \setminus g(I) \subseteq f(g(I))$. 
\end{thm}

Using the above notation, we refer to the set $\mathcal{C}:=\{ f(g(I)) \cup g(I) : I \in \mathcal{I}(\mathcal{H})$\} as a set of \emph{containers} and the $g(I) \in \mathcal{S}$ as \emph{fingerprints}.%

\smallskip

Throughout the paper, when we consider $r$-tuples of sets, the $r$-tuples are \emph{always} ordered.
For two $r$-tuples of sets $(I_1,\dots,I_r)$ and $(J_1,\dots,J_r)$ we write $(I_1,\dots,I_r) \subseteq (J_1,\dots,J_r)$ if $I_x \subseteq J_x$ for each $x \in [r]$. 
We write $(I_1,\dots,I_r)\cup (J_1,\dots,J_r):=(I_1 \cup J_1,\dots,I_r\cup J_r)$.

If $\mathcal X$ is a collection of sets then we write $\mathcal X^r$ for the collection of $r$-tuples $(X_1, \dots, X_r)$ so that $X_i \in \mathcal X$ for all $1\leq i \leq r$.
So for example, $\mathcal{P}([n])^r$ denotes the collection  of all $r$-tuples $(X_1, \dots, X_r)$ so that $X_i \subseteq [n]$ for all $1\leq i \leq r$.
 We write $ij$ to denote the pair $\{i,j\}$. For a hypergraph $\mathcal{H}$ 
define $$\mathcal{I}_r(\mathcal{H}):=\bigg\{(I_1,\dots,I_r) \in \mathcal{P}(V(\mathcal{H}))^r: I_x \in \mathcal{I}(\mathcal{H}) \text{ and } I_i \cap I_j = \emptyset \text{ for all } x \in [r], ij \in \binom{[r]}{2}\bigg\}.$$

Whereas Theorem~\ref{BMS} provides a set of containers for the independent sets of a hypergraph, the following proposition is an analogous result for the $r$-tuples of disjoint independent sets of a hypergraph. 
It is a straightforward consequence of Theorem~\ref{BMS}.

\begin{prop}\label{multi}
For every $k,r \in \mathbb{N}$ and all positive $c$ and $\eps$, there exists a positive constant $C$ such that the following holds. Let $\mathcal{H}$ be a $k$-uniform hypergraph and let $\mathcal{F} \subseteq \mathcal{P}(V(\mathcal{H}))$ be an increasing family of sets such that $|A| \geq \eps v(\mathcal{H})$ for all $A \in \mathcal{F}$. Suppose that $\mathcal{H}$ is $(\mathcal{F},\eps)$-dense and $p \in (0,1)$ is such that, for every $\ell \in [k]$, 
$$\Delta_{\ell}(\mathcal{H}) \leq c \cdot p^{\ell-1} \frac{e(\mathcal{H})}{v(\mathcal{H})}.$$ 
Then there exists a family $\mathcal{S}_r \subseteq \mathcal{I}_r(\mathcal{H})$ 
and functions $f:\mathcal{S}_r \to (\overline{\mathcal{F}})^r$ and $g:\mathcal I_r(\mathcal H) \to \mathcal S_r$ such that the following conditions hold:
\begin{itemize}
\item[(i)] If $(S_1,\dots, S_r) \in \mathcal S_r$ then $\sum |S_i|\leq Cp \cdot v(\mathcal{H})$;
\item[(ii)] for every $(I_1,\dots,I_r) \in \mathcal{I}_r(\mathcal{H})$, we have  that $S \subseteq (I_1,\dots,I_r) \subseteq S\cup f(S)$ where $S:=g( I_1,\dots, I_r)$.
\end{itemize} 
\end{prop}
\begin{proof}
Apply Theorem~\ref{BMS} with $k,c,\eps$ to obtain a positive constant $C_1$. Let $C:=rC_1$. We will show that $C$ has the required properties. Let $\mathcal{H}$ be a $k$-uniform hypergraph which together with a set $\mathcal{F} \subseteq \mathcal{P}(V(\mathcal{H}))$ satisfies the hypotheses of Proposition~\ref{multi}. 
Since $\mathcal{H}$, $\mathcal{F}$ also satisfy the hypotheses of Theorem~\ref{BMS}, there exists a family $\mathcal{S} \subseteq \binom{V(\mathcal{H})}{\leq C_1p\cdot v(\mathcal{H})}$ and functions $f':\mathcal{S} \to \overline{\mathcal{F}}$ and $g':\mathcal{I}(\mathcal{H}) \to \mathcal{S}$ such that for every $I \in \mathcal{I}(\mathcal{H})$ we have $g'(I) \subseteq I$ and $I \setminus g'(I) \subseteq f'(g'(I))$. Define 
$$\mathcal{S'}:=\{ S \in \mathcal{S}: \text{ there exists } I \in \mathcal{I}(\mathcal{H}) \text{ such that } g'(I)=S\},$$ and
$$\mathcal{S}_r:=\bigg\{(S_1,\dots,S_r) \in \mathcal{P}(V(\mathcal{H}))^r: S_x \in \mathcal{S'} \text{ and } S_i \cap S_j = \emptyset \text{ for all } x \in [r], ij \in \binom{[r]}{2}\bigg\}.$$

Let $(S_1,\dots,S_r) \in \mathcal{S}_r$. First note that $$ \sum_{x\in [r]} |S_x| \leq C_1 r \cdot p v(\mathcal{H}) = Cp \cdot v(\mathcal{H}),$$ so (i) holds. 
Also since $S_x \in \mathcal{S'}$ for all $x \in [r]$, we have $S_x \in \mathcal{I}(\mathcal{H})$ and so by definition of $\mathcal{S}_r$ we have $\mathcal{S}_r \subseteq \mathcal{I}_r(\mathcal{H})$. 

Consider any $(S_1,\dots,S_r) \in \mathcal{S}_r$ and any $(I_1,\dots,I_r) \in \mathcal{I}_r(\mathcal{H})$.
Define $f:\mathcal{S}_r \to (\overline{\mathcal{F}})^r$ by  setting $f(S_1,\dots,S_r):=(f'(S_1),\dots,f'(S_r))$ and define $g:\mathcal{I}_r(\mathcal{H}) \to \mathcal{S}_r$ by setting $g(I_1,\dots,I_r):=(g'(I_1),\dots,g'(I_r))$.

Note that since $f'(S_x) \in \overline{\mathcal{F}}$, $g'(I_x) \in \mathcal{S'}$ and $g'(I_i) \cap g'(I_j) = \emptyset$ for all $x \in [r]$ and $ij \in \binom{[r]}{2}$, we do indeed have $(f'(S_1),\dots,f'(S_r)) \in (\overline{\mathcal{F}})^r$ and $(g'(I_1),\dots,g'(I_r)) \in \mathcal{S}_r$.

Now for (ii), since $g'(I_x) \subseteq I_x$ and $I_x \setminus g'(I_x) \subseteq f'(g'(I_x))$ for all $x \in [r]$, we have
$g(I_1,\dots,I_r)=(g'(I_1),\dots,g'(I_r)) \subseteq (I_1,\dots,I_r)$. 
Since $f(g(I_1,\dots,I_r))=(f'(g'(I_1)),\dots,f'(g'(I_r)))$ we also have $(I_1,\dots,I_r) \subseteq f(g(I_1,\dots,I_r)) \cup g(I_1,\dots,I_r)$ as required. 
\end{proof}


In all of our applications of the container method, we will in fact apply the following asymmetric version of Proposition~\ref{multi}. In particular, in the proof of e.g.~Theorem~\ref{ramres}, instead of considering tuples of disjoint independent sets from the \emph{same} hypergraph $\mathcal H$, we are actually concerned with disjoint independent sets from \emph{different} hypergraphs but which have the \emph{same vertex set}:
For all $i \in [r]$, let $\mathcal{H}_i$ be a $k_i$-uniform hypergraph, each on the same vertex set $V$, and define  $\mathcal{I}(\mathcal{H}_1,\ldots,\mathcal{H}_r)$ to be the set of all $r$-tuples $(I_1,\ldots,I_r) \in \prod_{i \in [r]}\mathcal{I}(\mathcal{H}_i)$ such that $I_i \cap I_j = \emptyset$ for all $1 \leq i < j \leq r$.

We omit the proof of Proposition~\ref{asymmulti} since it follows from Theorem~\ref{BMS} as in the proof of Proposition~\ref{multi}.

\begin{prop}\label{asymmulti}
For every $r,k_1,\ldots,k_r \in \mathbb{N}$ with $k_i \geq 2$ for all $i \in [r]$, and all $c,\eps > 0$, there exists a positive constant $C$ such that the following holds. For all $i \in [r]$, let $\mathcal{H}_i$ be a $k_i$-uniform hypergraph, each on the same vertex set $V$.
For all $i \in [r]$, let $\mathcal{F}_i \subseteq \mathcal{P}(V)$ be an increasing family of sets such that $|A| \geq \eps |V|$ for all $A \in \mathcal{F}_i$. 
Suppose that each $\mathcal{H}_i$ is $(\mathcal{F}_i,\eps)$-dense.
Further suppose
$p \in (0,1)$ is such that, for every $i \in [r]$ and $\ell \in [k_i]$, 
$$\Delta_{\ell}(\mathcal{H}_i) \leq c \cdot p^{\ell-1} \frac{e(\mathcal{H}_i)}{|V|}.$$ 
Then there exists a family $\mathcal{S}_r \subseteq \mathcal{I}(\mathcal{H}_1,\ldots,\mathcal{H}_r) $ and  functions $f:\mathcal{S}_r \to \prod_{i \in [r]}\overline{\mathcal{F}_i}$ and $g:\mathcal{I}(\mathcal{H}_1,\ldots,\mathcal{H}_r) \to \mathcal S_r$
such that the following conditions hold:
\begin{itemize}
\item[(i)] If $(S_1,\dots, S_r) \in \mathcal S_r$ then $\sum |S_i|\leq Cp |V|$;
\item[(ii)] for every $(I_1,\dots,I_r) \in \mathcal{I}(\mathcal{H}_1,\ldots,\mathcal{H}_r)$, we have  that $S \subseteq (I_1,\dots,I_r) \subseteq S\cup f(S)$ where $S:=g( I_1,\dots, I_r)$.
\end{itemize}
\end{prop}
\section{Applications of the container method to \texorpdfstring{$(\LL,r)$}{(L,r)}-free sets}\label{sec4}

\COMMENT{At the moment we don't have a $0$-statement for the case of $r=1$. But probably that is easy to get that resilience is less then $\eps$ in this case:
 If $p \leq cn^{-1/m(A)}$ and $r \geq 2$ then Theorem~\ref{radores0} implies $\res([n]_p,(\LL,r)\text{-Rado})=0$. If $r=1$ then the expected number of solutions to $\LL$ in $[n]_p$ (which is $\Omega(n^{k-\ell} p^k$)) is significantly less than the expected size of $[n]_p$ (which is $np$). So there exists a subset of size $(1-o(1))|[n]_p|$ with no monochromatic solutions to $\LL$. So $\res([n]_p,(\LL,r)\text{-Rado}) \leq \eps |[n]_p|$.}

In this section we will prove Theorems~\ref{radores1new} and~\ref{radonumnew} by using the container theorem for $r$-tuples of disjoint independent sets, applied with irredundant partition regular matrices $A$. Suppose that we have a $k$-uniform hypergraph $\mathcal{H}$ whose vertex set is a subset of $\mathbb N$ and where the edges correspond to the $k$-distinct solutions of $\LL$. Then in this setting, an $(\LL,r)$-free set is precisely an $r$-tuple of disjoint independent sets in $\mathcal{H}$.

Theorems~\ref{radores1new} and~\ref{radonumnew} will be deduced from a container theorem, Theorem~\ref{matcontainer}, which in turn follows from Proposition~\ref{asymmulti}.
Theorem~\ref{matcontainer} actually holds for a class of irredundant matrices of which partition regular matrices are a subclass. Let $(*)$ be the following matrix property: 
\begin{itemize}
\item[$(*)$] Under Gaussian elimination $A$  does not have any row which consists of precisely two non-zero rational entries. 
\end{itemize}
Call an integer matrix $A$ (and the corresponding system of linear equations $\LL$) \emph{$r$-regular} if all $r$-colourings of $\mathbb{N}$ yield a monochromatic solution to $\LL$.  Observe that a matrix is $r$-regular for all $r \in \mathbb{N}$ if and only if it is partition regular.
As outlined in the next subsection, given any $r \geq 2$, all irredundant $r$-regular matrices $A$ satisfy $(*)$.
We will in fact prove stronger versions of  Theorems~\ref{radores1new} and~\ref{radonumnew} that 
consider irredundant matrices with property $(*)$.

These general results also consider `asymmetric' Rado properties:
Suppose that $\LL_i$ is a system of linear equations for each $1\leq i \leq r$ (and, here and elsewhere, $A_i$ is the matrix such that $\LL_i = \LL(A_i)$).
 We say a set $X\subseteq \mathbb N$ is \emph{$(\LL_1,\dots,\LL_r)$-free} if there is an $r$-colouring of $X$ such that there are no solutions to $\LL_i$ in $X$ in colour $i$ for every $i \in [r]$. 
Otherwise we say that $X$ is \emph{$(\LL_1,\dots,\LL_r)$-Rado}. 
We denote the size of the largest $(\LL_1,\dots,\LL_r)$-free subset of $[n]$ by $\mu(n,\LL_1,\dots,\LL_r)$. 

In general it is not known which systems of linear equations $\LL_1,\dots,\LL_r$ are
such that $\mathbb N$ is $(\LL_1,\dots,\LL_r)$-Rado. However, if each
 $\LL _i$ is an $r$-regular homogenous linear equation, then 
$\mathbb N$ is $(\LL_1,\dots,\LL_r)$-Rado (see~\cite[Theorem 9.19]{book}).

 We will prove the following strengthenings of Theorems~\ref{radores1new} and~\ref{radonumnew}.

\begin{thm}\label{radores1}
For all positive integers $r$, all irredundant full rank matrices $A_1, \dots, A_r$ which satisfy $(*)$ with $m(A_1) \geq \dots \geq m(A_r)$, and all $\delta >0$, there exists a constant $C>0$ such that
$$
\lim_{n \rightarrow \infty} \mathbb{P}\left[ \frac{\res([n]_p,(\LL_1,\dots,\LL_r)\text{-Rado})}{|[n]_p|}= 1-\frac{\mu(n,\LL_1,\dots,\LL_r)}{n}\pm \delta   \right] = 1 \quad\text{ if } p > Cn^{-1/m(A_1)}.
$$  
\end{thm}

\begin{thm}\label{radonum}
For all positive integers $r$, all irredundant full rank matrices $A_1, \dots, A_r$ which satisfy $(*)$, there are $2^{\mu(n,\LL_1,\dots,\LL_r)+o(n)}$ $(\LL_1,\dots,\LL_r)$-free subsets of $[n]$.
\end{thm}

Given a system of linear equations $\LL$, a \emph{strongly $\LL$-free} subset of $[n]$ is a subset that contains no solution to $\LL$. Although this is not quite the same definition as $\LL$-free, we remark that
Theorem~\ref{radonum} implies a result of Green~\cite[Theorem 9.3]{G-R} in the case where $k \geq 3$\COMMENT{No mention of $k \geq 3$ in Green result, so think Green's method works for $k=2$ too}, on the number of strongly $\LL$-free subsets of $[n]$ for homogeneous linear equations $\LL$.

{\bf Additional note.}
As mentioned in the introduction, Spiegel~\cite{sp} independently proved the case $r=1$ of Theorem~\ref{radores1}.
(Note in~\cite{sp} this result is mentioned in terms of \emph{abundant} matrices $A$. That is every $\ell \times (k-2)$ submatrix of $A$ has rank $\ell$. But this is clearly equivalent to $(*)$ in the case of  irredundant full rank matrices.)

\subsection{Matrices which satisfy \texorpdfstring{$(*)$}{(*)}}
First we prove that irredundant partition regular matrices are a strict subclass of irredundant matrices which satisfy $(*)$.

Suppose that an irredundant matrix $A$ does not satisfy $(*)$. Then there exists a pair $ij \in \binom{[k]}{2}$ and non-zero rationals $\alpha, \beta$ such that for all solutions $(x_1,\dots,x_k)$ to $\LL$ we have $\alpha x_i=\beta x_j$. If $\alpha=\beta$ then no solution to $\LL$ is $k$-distinct and so $A$ is redundant, a contradiction. Otherwise, without loss of generality, assume that $\alpha > \beta >0$, and devise the following $2$-colouring of $\mathbb{N}$: greedily colour the numbers $\{1,2,3,...\}$ so that when colouring $x$, we always give it a different colour to $\beta x/\alpha$ (if $\beta x/\alpha \in \mathbb{N}$). Such a colouring ensures that no solution to $\LL$ is monochromatic, and so $A$ is not partition regular. 

Note that the converse is not true. An $\ell \times k$ matrix with columns $a^{(1)},\dots,a^{(k)}$ satisfies the \emph{columns property} if there is a partition of $[k]$, say $[k]=D_1 \cup \dots \cup D_t$ such that $$\sum_{i \in D_1} a^{(i)} =0$$ and for every $r \in [t]$ we have $$ \sum_{i \in D_r} a^{(i)} \in \langle a^{(j)} : j \in D_1 \cup \dots \cup D_{r-1} \rangle.$$ Rado's theorem~\cite{rado} states that a matrix is partition regular if and only if it satisfies the columns property. Now, for example $A:=\begin{pmatrix} 2 & 2 & -1 \\ \end{pmatrix}$ is irredundant and clearly satisfies $(*)$, and additionally does not have the columns property, so is not partition regular. 

The argument above actually implies that if an irredundant matrix $A$ is $2$-regular, then it satisfies $(*)$. 
So in the symmetric case, Theorems~\ref{radores1} and~\ref{radonum} consider \emph{all} pairs $(A,r)$ such that $A$ is an irredundant $r$-regular matrix and $r\geq 2$.

\COMMENT{ Given $r \in \mathbb{N}$, if we have a matrix which is not $r$-regular but satisfies $(*)$ and is irredundant, then note that Theorems~\ref{radores1} and~\ref{radonum} hold trivially, since $\mu(n,\LL,r)=n$.}

\subsection{Useful matrix lemmas}
Before we can prove our container result (Theorem~\ref{matcontainer}), we require some matrix lemmas. Note that all of these lemmas hold for irredundant matrices which satisfy $(*)$. As a consequence, Theorem~\ref{radores0} was actually implicitly proven for irredundant matrices which satisfy $(*)$, since in~\cite{random4}  the only necessity of the matrix being partition regular was so that the results stated below could be applied.

Recall the definition of $m(A)$ given by~(\ref{m(A)def}). Parts (i) and (ii) of the following proposition were verified for irredundant partition regular matrices by R\"odl and Ruci\'nski (see Proposition 2.2 in~\cite{random4}). In fact their result easily extends to matrices which satisfy $(*)$. We give the full proof for completeness, and add further facts ((iii)--(v)) which will be useful in the proof of Theorem~\ref{matcontainer}.

\begin{prop}\label{matmL}
Let $A$ be an $\ell \times k$ irredundant matrix of full rank $\ell$ which satisfies $(*)$. Then for every $W \subseteq [k]$, the following hold.
\begin{itemize}
\item[(i)]{If $|W|=1$, then $\rank(A_{\overline{W}})=\ell$.}
\item[(ii)]{If $|W|\geq 2$, then $\ell-\rank(A_{\overline{W}})+2 \leq |W|$.}
\item[(iii)]{If $|W|\geq 2$, then $$-|W|-\rank(A_{\overline{W}}) \leq -\ell-1-\frac{|W|-1}{m(A)}.$$}
\end{itemize}
Furthermore,
\begin{itemize}
\item[(iv)]{$k \geq \ell+2$;}
\item[(v)]{$m(A) > 1$.}
 \end{itemize}
\end{prop}

\begin{proof}
For (i), suppose that $\rank(A_{\overline{W}})=\ell-1$ for some $W \subseteq [k]$ with $|W|=1$. Since $A_{\overline{W}}$ is an $\ell \times (k-1)$ matrix of rank $\ell-1$, under Gaussian elimination it must contain a row of zeroes. Hence $A$ under Gaussian elimination contains a row with at most one non-zero entry. If there is a non-zero entry in this row, then there are no positive solutions to $\LL$, which contradicts $A$ being irredundant. If there are none, then $A$ does not have rank $\ell$, also a contradiction. 

For (ii) proceed by induction on $|W|$. Assume first that there is a $W \subseteq [k]$ with $|W|=2$, such that $\rank(A_{\overline{W}})<\ell$. Using a similar argument to (i), under Gaussian elimination $A$ contains a row with at most two non-zero entries. If there are two non-zero entries this contradicts $A$ satisfying $(*)$. Otherwise we again get a contradiction to either $A$ being irredundant or of rank $\ell$. Assume now that $|W|\geq 3$ and that the statement holds for $|W|-1$. The rank of a matrix drops by at most one when a column is deleted, hence the required inequality follows by induction.

For (iii), note that for $|W| \geq 2$, by definition we have $m(A) \geq (|W|-1)/(|W|-1+\rank(A_{\overline{W}})-\ell)$. This can be rearranged to give the required inequality. For (iv), by taking $W=[k]$ the result follows immediately from (ii). 
For (v), again take $W=[k]$. Then by definition~(\ref{m(A)def})  $m(A) \geq (k-1)/(k-\ell-1) >  1$, where the second inequality follows since the denominator is positive by (iv).
\end{proof}

The following supersaturation lemma follows easily from the (1-colour) removal lemma proved for integer matrices by Kr\'al', Serra and Vena (Theorem 2 in~\cite{ksv2}).


\begin{lemma}\label{matsupersat}
Fix $r \in \mathbb N$ and for each $i \in [r]$, let $A_i$ be an $\ell_i \times k_i$ integer matrix of rank $\ell_i$, and write $\mathcal{L}_i:=\mathcal{L}(A_i)$. For every $\delta >0$ there exist $n_0,\eps > 0$ with the following property. Suppose $n \geq n_0$ is an integer and $X \subseteq [n]$ is $r$-coloured, and $|X| \geq \mu(n,\LL_1,\dots,\LL_r)+\delta n$. Then there exists an $i \in [r]$ such that there are more than $\eps n^{k_i-\ell_i}$ $k_i$-distinct solutions to $\LL_i$ in colour $i$ in $X$.
\end{lemma}

Finally we need the following well known result (and a simple corollary of it), which gives a useful upper bound on the number of solutions to a system of linear equations. Note that in this lemma only, we \emph{do not} assume $A$ to be necessarily of full rank (as we will apply the result directly to matrices formed by deleting columns from our original matrix of full rank). 

\begin{lemma}\label{matupper}
 For an $\ell \times k$ matrix $A$ not necessarily of full rank, an $\ell$-dimensional integer vector $b$ and a set $X \subseteq [n]$, the system $A x=b$ has at most $|X|^{k-\rank(A)}$ solutions in $X$. 
\end{lemma}
\begin{proof}
Use Gaussian elimination to turn $A$ into echelon form. Now note that when picking a solution to $Ax=b$ in $X$ (where $x=(x_1,\dots,x_k))$, there are $|X|$ choices for $k-\rank(A)$ of the $x_i$ (the `free' variables), and the other $\rank(A)$ of the $x_i$ are immediately determined. Thus there are at most $|X|^{k-\rank(A)}$ solutions as required.
\end{proof}

\begin{corollary}\label{matupper2}
Consider an $\ell \times k$ matrix $A$ of rank $\ell$, a set $X \subseteq [n]$ and an integer $1\leq t \leq k$.
Fix distinct $y_1,\dots, y_t \in X$  and consider any $W=\{s_1,\dots, s_t\} \subseteq [k]$.  
The system $Ax=0$ has at most $|X|^{k-t-\rank(A_{\overline{W}})}$ solutions $(x_1,\dots, x_k)$ in $X$ 
for which $x_{s_j}=y_j$ for each $j \in [t]$.
 Moreover, if $A$ is irredundant and satisfies $(*)$ and $t=1$, then the system $Ax=0$ has at most $|X|^{k-\ell-1}$ solutions $(x_1,\dots, x_k)$ in $X$ for which $x_{s_1}=y_1$.
\end{corollary}

\begin{proof}
Write  $A=:(a_{ij})$. Consider  the system of linear equations $A_{\overline{W}}x'=b$ where, for each $r\in [\ell]$, the $r^{\mathrm{th}}$ term in $b$ is
$$b_r:=-\sum_{s_j \in W} a_{rs_j} y_j.$$
 Now by Lemma~\ref{matupper} the system of linear equations $A_{\overline{W}}x'=b$ has at most $|X|^{k-t-\rank(A_{\overline{W}})}$ solutions in $X$. The first part of the corollary then follows since all solutions $(x_1,\dots,x_k)$ to $Ax=b$ with $x_{s_j}=y_j$ for each $j \in [t]$,  rise from a solution $x'$ to $A_{\overline{W}}x'=b$. For the second part, if $A$ is irredundant and  satisfies $(*)$ and $t=1$, then by Proposition~\ref{matmL}(i), we have $\rank(A_{\overline{W}})=\ell$ and so the result follows.
\end{proof}

\subsection{A container theorem for tuples of \texorpdfstring{$\LL$}{L}-free sets}
Recall that an $\LL$-free set is simply an $(\LL,1)$-free set.
Let $\mathcal I(n,\LL_1,\dots,\LL_r)$ denote the set of all ordered $r$-tuples $(X_1,\dots,X_r)\in \mathcal P([n])^r$ so that each $X_i$ is $\LL_i$-free and $X_i\cap X_j=\emptyset$ for all distinct $i,j \in [r]$.
Note that any $(\LL_1,\dots,\LL_r)$-free subset $X$ of $[n]$ has a partition $X_1,\dots,X_r$ so that $(X_1,\dots,X_r)\in \mathcal I(n,\LL_1,\dots,\LL_r)$.
We now prove a container theorem for the elements of $\mathcal I(n,\LL_1,\dots,\LL_r)$. 

\begin{thm}\label{matcontainer}
Let $r \in \mathbb N$ and $0<\delta <1$. For each $i \in [r]$ let $A_i$ be an $\ell_i \times k_i$ irredundant matrix of full rank $\ell_i$ which satisfies $(*)$, and suppose that $m(A_1) \geq \dots \geq m(A_r)$. Then there exists $D>0$ such that the following holds. For all $n \in \mathbb{N}$, there is a collection $\mathcal S_r\subseteq
\mathcal P([n])^r$  and a function $f:\mathcal S_r \rightarrow \mathcal P([n])^r $  such that:

\begin{itemize}
\item[(i)]{For all $(I_1,\dots,I_r)\in \mathcal I(n,\LL_1,\dots,\LL_r)$, there exists $S \in \mathcal S_r$ such that $S \subseteq (I_1,\dots,I_r) \subseteq f(S)$.}
\item[(ii)] If $(S_1,\dots,S_r) \in \mathcal S_r$ then $\sum_ {i \in [r]}|S_i|\leq Dn^{\frac{m(A_1)-1}{m(A_1)}}$.
\item[(iii)] Every $S \in \mathcal S_r$ satisfies $S \in \mathcal I(n,\LL_1,\dots,\LL_r)$.
\item[(iv)] Given any $S=(S_1,\dots,S_r) \in \mathcal S_r$, write $f(S)=:(f(S_1),\dots,f(S_r))$. Then
\begin{itemize}
\item[(a)] for each $1\leq i \leq r$, $f(S_i)$ contains at most $\delta n^{k_i-\ell_i}$ $k_i$-distinct  solutions to $\LL_i$; and
\item[(b)] $|\cup _{i \in [r]} f(S_i)|\leq \mu(n,\LL_1,\dots,\LL_r)+\delta n$.
\end{itemize}
\end{itemize}
\end{thm}
We emphasise that (iv)(b) does not necessarily guarantee $\sum  _{i \in [r]} |f(S_i)|\leq \mu(n,\LL_1,\dots,\LL_r)+\delta n$. Rather it ensures at most $\mu(n,\LL_1,\dots,\LL_r)+\delta n$ elements of $[n]$ appear in at least one of the co-ordinates of $f(S)$.
This property is crucial for our applications.

\begin{proof}
First note that since each of the matrices $A_i$ are irredundant, a result of Janson and Ruci\'nski~\cite{JR} implies that there exists a constant $d>0$ such that, for each $i \in [r]$, there are at least $dn^{k_i-\ell_i}$ $k_i$-distinct solutions to $\LL_i$ in $[n]$. 

Note that it suffices to prove the theorem in the case when $0 < \delta < d$. Also, it suffices to prove the theorem when $n$ is sufficiently large; otherwise we can set $\mathcal S_r$ to be $\mathcal I(n,\LL_1,\dots,\LL_r)$; set $f$ to be the identity function and choose $D$ to be large. 

Let $0<\delta <d$ and $r \in \mathbb N$ be given and apply Lemma~\ref{matsupersat} to obtain $n_0,\eps>0$. Without loss of generality we may assume $\eps \leq \delta$. Define $k:=\max k_i$ and 
let $$\eps':=\frac{\eps}{2} \quad\text{and}\quad c:=\frac{k!}{\eps'}.$$
 Apply Proposition~\ref{asymmulti} with parameters $r, k_1, \dots, k_r, c,\eps'$ playing the roles of $r, k_1, \dots, k_r, c,\eps$ respectively to obtain $D>0$. Increase $n_0$ if necessary so that $0<1/n_0 \ll 1/D, 1/k_1, \dots, 1/k_r, 1/r,\eps, \delta$ and let $n \geq n_0$ be an integer. 
 
For each $i \in [r]$ let $\mathcal{H}_{n,i}$ be the hypergraph with $V(\mathcal{H}_{n,i}):=[n]$ and an edge set which consists of all $k_i$-distinct solutions to $\LL_i$ in $[n]$. Observe that $\mathcal{H}_{n,i}$ is $k_i$-uniform and an independent set in $\mathcal{H}_{n,i}$ is an $\LL_i$-free set.

For each $i \in [r]$ we define $\mathcal{F}_{n,i}:=\{F \subseteq V(\mathcal{H}_{n,i}): e(\mathcal{H}_{n,i}[F]) \geq \eps' e(\mathcal{H}_{n,i})\}$. Note that since $\eps'<d$, we have 
\begin{align}\label{edgemat}
\eps'n^{k_i-\ell_i} \leq e(\mathcal{H}_{n,i}).
\end{align}
We claim that $\mathcal{H}_{n,i}$ and $\mathcal{F}_{n,i}$ satisfy the hypotheses of Proposition~\ref{asymmulti} with parameters chosen as above with $$p=p(n):=n^{-1/m(A_1)}.$$

Clearly $\mathcal{F}_{n,i}$ is increasing and $\mathcal{H}_{n,i}$ is $(\mathcal{F}_{n,i},\eps')$-dense. By Lemma~\ref{matupper}, a set $F \subseteq V(\mathcal{H}_{n,i})$ contains at most $|F|^{k_i-\ell_i}$ solutions to $\LL_i$ (so $e(\mathcal{H}_{n,i}[F]) \leq |F|^{k_i-\ell_i}$). 
Hence for all $F \in \mathcal{F}_{n,i}$, we have $$|F| \geq e(\mathcal{H}_{n,i}[F])^{\frac{1}{k_i-\ell_i}} \geq (\eps' e(\mathcal{H}_{n,i}))^{\frac{1}{k_i-\ell_i}} \stackrel{(\ref{edgemat})}{\geq} ((\eps')^2 n^{k_i-\ell_i})^{\frac{1}{k_i-\ell_i}} \geq \eps' n$$ where the last inequality follows by Proposition~\ref{matmL}(iv).

For each $j \in [k_i]$, we wish to bound the number of hyperedges containing some $\{y_1,\dots,y_j\}\subseteq V(\mathcal{H}_{n,i})$.
Suppose $(x_1,\dots, x_{k_i})$ is a $k_i$-distinct solution to $\LL_i$ so that $\{y_1,\dots, y_j\}\subseteq \{x_1,\dots,x_{k_i}\}$.
There are ${k_i}!/({k_i}-j)!$ choices for picking the $j$ roles the $y_i$ play in $(x_1,\dots, x_{k_i})$. Let $W$ be one such choice for the set of indices of the $x_a$ used by $\{y_1,\dots,y_j\}$.
In this case, Corollary~\ref{matupper2} implies there are at most $n^{{k_i}-j-\rank((A_i)_{\overline{W}})}$ such solutions to $\LL_i$, and if $j=1$, there are at most $n^{k_i-\ell_i-1}$ such solutions.
So for $j=1$ this yields 
\begin{align*}
\deg_{\mathcal{H}_{n,i}}(y_1) \leq k_i n^{k_i-\ell_i-1} \stackrel{(\ref{edgemat})}{\leq} \frac{k_i}{\eps'} \frac{e(\mathcal{H}_{n,i})}{v(\mathcal{H}_{n,i})} \leq  c \frac{e(\mathcal{H}_{n,i})}{v(\mathcal{H}_{n,i})}.
\end{align*}
 For $j\geq 2$, by Proposition~\ref{matmL}(iii) we have $k_i-j-\rank((A_i)_{\overline{W}}) \leq k_i-\ell_i-1-(j-1)/m(A_i)$. Also $m(A_1) \geq m(A_i)$ for all $i \in [r]$ and hence we have 
\begin{align*}
\deg_{\mathcal{H}_{n,i}}(\{y_1,\dots,y_j\}) & \leq k_i! n^{k_i-\ell_i-1-\frac{j-1}{m(A_i)}} \leq k_i! n^{k_i-\ell_i-1-\frac{j-1}{m(A_1)}} \\
& \leq \frac{k_i!}{\eps'} p^{j-1} \frac{e(\mathcal{H}_{n,i})}{v(\mathcal{H}_{n,i})} \leq cp^{j-1} \frac{e(\mathcal{H}_{n,i})}{v(\mathcal{H}_{n,i})}.
\end{align*}
Since $\{y_1,\dots,y_j\}$ was arbitrary, we therefore have $\Delta_j(\mathcal{H}_{n,i}) \leq cp^{j-1}e(\mathcal{H}_{n,i})/v(\mathcal{H}_{n,i})$, as required. 
We have therefore shown that $\mathcal{H}_{n,i}$ and $\mathcal{F}_{n,i}$ satisfy the hypotheses of Proposition~\ref{asymmulti} for all $i \in [r]$. 

\medskip

Then Proposition~\ref{asymmulti} implies that there exists a family $\mathcal{S}_r \subseteq \prod_{i \in [r]}\mathcal{P}(V(\mathcal{H}_{n,i})) = \mathcal{P}([n])^r$ and  functions $f':\mathcal{S}_r \to \prod_{i \in [r]}\overline{\mathcal{F}_{n,i}}$ and $g:\mathcal{I}(\mathcal{H}_{n,1},\ldots,\mathcal{H}_{n,r}) \to \mathcal S_r$
such that the following conditions hold:
\begin{itemize}
\item[(a)] If $(S_1,\dots, S_r) \in \mathcal S_r$ then $\sum_{i\in [r]} |S_i|\leq Dpn$;
\item[(b)] every $S \in \mathcal{S}_r$ satisfies $S \in \mathcal{I}(\mathcal{H}_{n,1},\ldots,\mathcal{H}_{n,r})$;
\item[(c)] for every $(I_1,\dots,I_r) \in \mathcal{I}(\mathcal{H}_{n,1},\ldots,\mathcal{H}_{n,r})$, we have  that $S \subseteq (I_1,\dots,I_r) \subseteq S\cup f'(S)$, where $S:=g( I_1,\dots, I_r)$.
\end{itemize}

Note that $\mathcal{I}(\mathcal{H}_{n,1},\ldots,\mathcal{H}_{n,r})=\mathcal I(n,\LL_1,\dots,\LL_r)$. For each $S\in \mathcal S_r$, define 
$$f(S):=S\cup f'(S).$$
So $f:\mathcal S_r \rightarrow \mathcal P([n])^r $. Thus, (a)--(c) immediately imply that (i)--(iii)  hold.

Given any $S=(S_1,\dots,S_r)\in \mathcal S_r$ write $f(S)=:(f(S_1),\dots,f(S_r))$ and  $f'(S)=:(f'(S_1),\dots,f'(S_r))$. (Note the slight abuse of the use of the $f$ and $f'$ notation here.)
By definition of $\mathcal{F}_{n,i}$ any $F \in \overline{\mathcal{F}_{n,i}}$ contains at most $\eps ' n^{k_i-\ell_i}$ $k_i$-distinct solutions to $\LL_i$. 
By Corollary~\ref{matupper2}, the number of $k_i$-distinct solutions to $\LL_i$ in $[n]$ that use at least one element from $S_i$ is at most $k_i n^{k_i-\ell_i-1} |S_i|$.
Further,
$$k_i n^{k_i-\ell_i-1} |S_i| \leq k_i Dpn^{k_i-\ell_i} \leq \eps ' n^{k_i-\ell_i}.$$ 
Here, the first inequality holds by (a), and the second  since $p=n^{-1/m(A_1)}$ and $m(A_1)>0$ by Proposition~\ref{matmL}(v). 
Thus, in total $f(S_i)=S_i\cup f'(S_i)$ contains at most $2\eps ' n^{k_i-\ell_i}\leq \delta n^{k_i-\ell_i}$  $k_i$-distinct solutions to $\LL_i$, so (iv)($a$) holds.

In fact, the argument above
 implies that there is an $r$-colouring of the set 
$\cup _{i \in [r]} f(S_i)$ so that there are at most $2\eps' n^{k_i-\ell_i}=\eps  n^{k_i-\ell_i}$ $k_i$-distinct solutions to $\LL_i$ in colour $i$, in $\cup _{i \in [r]} f(S_i)$.
  Hence, Lemma~\ref{matsupersat} ensures (iv)($b$), as desired. 
\end{proof}

\subsection{The number of \texorpdfstring{$(\LL_1,\dots,\LL_r)$}{(\LL_1,\dots,\LL_r)}-free subsets of \texorpdfstring{$[n]$}{[n]}}
Our first application of Theorem~\ref{matcontainer} yields an enumeration result (Theorem~\ref{radonum}) for the number of $(\LL_1,\dots,\LL_r)$-free subsets of $[n]$.


\begin{proof}[Proof of Theorem~\ref{radonum}]
By definition of $\mu(n,\LL_1,\dots,\LL_r)$ there are at least $2^{\mu(n,\LL_1,\dots,\LL_r)}$ $(\LL_1,\dots,\LL_r)$-free subsets of $[n]$. So it suffices to prove the upper bound. 

For this, note that we may assume $n$ is sufficiently large.
Let $0<\delta<1$ be arbitrary and let $D>0$ be obtained from Theorem~\ref{matcontainer} applied to $A_1,\dots,A_r$ with parameter $\delta$. We obtain a collection $\mathcal S_r$ and function $f$ as in Theorem~\ref{matcontainer}.
Consider any $(\LL_1,\dots,\LL_r)$-free subset $X$ of $[n]$. Note that $X$ has a partition $X_1,\dots,X_r$ so that $(X_1,\dots,X_r)\in \mathcal I(n,\LL_1,\dots,\LL_r)$. So by Theorem~\ref{matcontainer}(i) this means there is some $S=(S_1,\dots,S_r)\in \mathcal S_r$
so that $X \subseteq \cup _{i \in [r]} f(S_i)$. 

Further, given any $S=(S_1,\dots,S_r)\in \mathcal S_r$, we have that $|\cup _{i \in [r]} f(S_i)|\leq \mu(n,\LL_1,\dots,\LL_r)+\delta n$. Thus, each such $\cup _{i \in [r]} f(S_i)$ contains at most $2^{\mu(n,\LL_1,\dots,\LL_r)+\delta n}$ $(\LL_1,\dots,\LL_r)$-free subsets of $[n]$.
Note that, by Theorem~\ref{matcontainer}(ii),
$$|\mathcal S_r|\leq \left(\sum_{s=0}^{Dn^{\frac{m(A_1)-1}{m(A_1)}}} \binom{n}{s}\right)^r<2^{\delta n},$$
where the last inequality holds since $n$ is sufficiently large.%

Altogether, this implies that the number of $(\LL_1,\dots,\LL_r)$-free subsets  of $[n]$ is at most
$$2^{\delta n} \times 2^{\mu(n,\LL_1,\dots,\LL_r)+\delta n}= 2^{\mu(n,\LL_1,\dots,\LL_r)+2\delta n}.$$
Since the choice of $0<\delta<1$ was arbitrary this proves the theorem.
\end{proof}

\subsection{The resilience of being \texorpdfstring{$(\LL_1,\dots,\LL_r)$}{(\LL_1,\dots,\LL_r)}-Rado}  
Recall that the resilience of $S$ with respect to $\mathcal{P}$, $\res(S,\mathcal{P})$, is the minimum number $t$ such that by deleting $t$ elements from $S$, one can obtain a set not having $\mathcal{P}$. In this section we will determine $\res([n]_p,(\LL_1,\dots,\LL_r)\text{-Rado})$ for irredundant matrices $A_1,\dots,A_r$ which satisfy $(*)$. We now use Theorem~\ref{matcontainer} to deduce Theorem~\ref{radores1}.

\begin{proof}[Proof of Theorem~\ref{radores1}]
Let $0<\delta<1$, $r \in \mathbb N$ and $A_1,\dots,A_r$ be matrices as in the statement of the theorem. 
Given $n$, if $p > n^{-1/m(A_1)}$ then since $m(A_1)>1$ by  Proposition~\ref{matmL}(v), Proposition~\ref{chernoff} implies that, w.h.p.,
\begin{align}\label{concmat}
|[n]_p| = \left(1 \pm \frac{\delta}{4} \right) pn.
\end{align}
We first show that 
$$
\lim_{n \rightarrow \infty} \mathbb{P}\left[ \frac{\res([n]_p,(\LL_1,\dots,\LL_r)\text{-Rado})}{|[n]_p|}\leq 1-\frac{\mu(n,\LL_1,\dots,\LL_r)}{n}+ \delta   \right] = 1 \quad\text{ if } p > n^{-1/m(A_1)}.
$$  
For this, we must show that the probability of the event that there exists a set $S \subseteq [n]_p$ such that $|S| \geq (\mu(n,\LL_1,\dots,\LL_r)/n-\delta)|[n]_p|$ and $S$ is $(\LL_1,\dots,\LL_r)$-free, tends to one as $n$ tends to infinity. This indeed follows: 
Let $T$ be an $(\LL_1,\dots,\LL_r)$-free subset of $[n]$ of maximum size $\mu(n,\LL_1,\dots,\LL_r)$. Then, by Proposition~\ref{chernoff}, w.h.p. we have $|T \cap [n]_p|=(\mu(n,\LL_1,\dots,\LL_r)/n \pm \delta)|[n]_p|$, and $T \cap [n]_p$ is $(\LL_1,\dots,\LL_r)$-free, as required.\COMMENT{This even works in
$\mu$ is sublinear}

For the remainder of the proof, we will focus on the lower bound, namely that there exists $C>0$ such that whenever $p>Cn^{-1/m(A_1)}$,
\begin{align}\label{lab}
 \mathbb{P}\left[ \res([n]_p,(\LL_1,\dots,\LL_r)\text{-Rado}) \geq \left( 1-\frac{\mu(n,\LL_1,\dots,\LL_r)}{n} - \delta \right) |[n]_p| \right] \to 1 \quad\text{ as } n \to \infty.
\end{align}

Suppose $n$ is sufficiently large.
Apply Theorem~\ref{matcontainer} with parameters $r,\delta/8,A_1,\dots,A_r$ to obtain $D>0$,  a collection $\mathcal S_r\subseteq \mathcal P([n])^r$ and a function $f$ satisfying (i)--(iv). Now choose $
C$ such that $0<  1/C \ll 1/D, \delta,1/r$. 
Let $p \geq Cn^{-1/m(A_1)}$.

Since (\ref{concmat}) holds with high probability, 
to prove (\ref{lab}) holds it suffices to show that the probability $[n]_p$ contains an $(\LL_1,\dots,\LL_r)$-free subset of size at least $(\frac{\mu(n,\LL_1,\dots,\LL_r)}{n} +\delta /2 )np$ tends to zero as $n$ tends to infinity.

Suppose  that $[n]_p$ 
does contain an $(\LL_1,\dots,\LL_r)$-free subset $I$ of size at least $(\frac{\mu(n,\LL_1,\dots,\LL_r)}{n} +\delta /2)np$. Note that $I$ has a partition $I_1,\dots, I_r$ so that $(I_1,\dots, I_r)\in \mathcal I(n,\LL_1,\dots,\LL_r)$.
Further, there is some $S=(S_1,\dots,S_r) \in \mathcal S_r$ such that $S\subseteq (I_1,\dots, I_r) \subseteq f(S)$. 
Thus, $[n]_p$ must contain $\cup _{i\in [r]}S_i$ as well as at least $(\frac{\mu(n,\LL_1,\dots,\LL_r)}{n} +\delta /4)np$ elements from $\left(\cup _{i\in [r]}f(S_i)\right )\setminus \left (\cup _{i\in [r]}S_i \right)$. (Note here we are using that $|\cup _{i\in [r]}S_i|\leq \delta np/4$, which holds by Theorem~\ref{matcontainer}(ii) and since $0<  1/C \ll 1/D, \delta$.) Writing $s:=|\cup _{i\in [r]}S_i|$,
the probability $[n]_p$ contains $\cup _{i\in [r]}S_i$ is $p^s$. 
Note that $|\left(\cup _{i\in [r]}f(S_i)\right )\setminus \left (\cup _{i\in [r]}S_i \right)|\leq \mu (n,\LL_1,\dots,\LL_r) +\delta n/8$ by Theorem~\ref{matcontainer}(iv)(b).
So  by the first part of
Proposition~\ref{chernoff},
the probability $[n]_p$ contains at least  $(\frac{\mu(n,\LL_1,\dots,\LL_r)}{n} +\delta /4)np$ elements from $\left(\cup _{i\in [r]}f(S_i)\right )\setminus \left (\cup _{i\in [r]}S_i \right )$, is 
 at most $\exp (-\delta ^2 np/256)$.

Write $N:=n^{(m(A_1)-1)/m(A_1)}$ and $\gamma := \delta ^2 /256$.
Given some $0\leq s\leq DN$, there are at most $r^s \binom{n}{s}$ elements $(S_1,\dots,S_r) \in \mathcal S_r$ such that $|\cup _{i\in [r]}S_i| =s$. Indeed, this follows since there are $r^s$ ways to partition a set of size $s$ into $r$ classes. (Note we only need to consider $s \leq DN$ by Theorem~\ref{matcontainer}(ii).)
Thus, the probability $[n]_p$ 
does contain an $(\LL_1,\dots,\LL_r)$-free subset $I$ of size at least $(\frac{\mu(n,\LL_1,\dots,\LL_r)}{n} +\delta /2)np$ is at most
\begin{align*}
\sum _{s=0} ^{DN} r^s \binom{n}{s} \cdot  p^s \cdot e^{-\gamma np}  & \leq  (DN+1)(rp)^{DN}\binom{n}{DN} e^{-\gamma np} 
 \leq (DN+1)\left (\frac{repn }{DN}\right )^{DN} e^{-\gamma np} \\ & \leq  (DN+1)\left (\frac{reC }{D}\right )^{DN} e^{-\gamma  CN} 
\leq e^{\gamma CN/2} e^{-\gamma CN} = e^{-\gamma CN/2} 
\end{align*}
which tends to zero as $n$ tends to infinity. This completes the proof.
\end{proof}

\subsection{The size of the largest \texorpdfstring{$(\LL,r)$}{(L,r)}-free set}\label{secaw}
Both as a natural question in itself, and in 
 light of Theorems~\ref{radores1} and~\ref{radonum}, it is of interest to obtain  good bounds on $\mu(n,\LL_1,\dots,\LL_r)$. For the rest of this section consider the symmetric case ($A:=A_1=\dots=A_r$) and assume that $A$ is a $1 \times k$ matrix, i.e.~we are interested in solutions to a linear equation $a_1 x_1+\dots+a_k x_k=0$. Such $\LL$ are called \emph{translation-invariant} if the coefficients $a_i$ sum to zero. It is known that $\mu(n,\LL,1)=o(n)$ if $\LL$ is translation-invariant and  $\mu(n,\LL,1)=\Omega(n)$ otherwise (see~\cite{ruzsa}). Determining exact bounds remains open in many cases, famously including \emph{progression-free sets} (where $\LL$ is $x+y=2z$). See~\cite{bloom,elk,greenwolf} for the state-of-the-art lower and upper bounds for this case. 

Call $S \subseteq [n]$ \emph{strongly $(\LL,r)$-free} if there exists an $r$-colouring of $S$ which contains no monochromatic solutions to $\LL$ of any type  (that is, solutions are not required to be $k$-distinct). Define $\mu^{*}(n,\LL,r)$ to be the size of the largest strongly $(\LL,r)$-free subset $S \subseteq [n]$. Note that for any density regular matrix $A$, $(x,\dots,x)$ is a solution to $\LL$ for all $x \in [n]$ (as observed by  Frankl, Graham and R\"odl~\cite[Fact 4]{fgr}) and so we have $\mu^{*}(n,\LL,r)=0$. (Note that this result implies that all density regular $1 \times k$ matrices give rise to an equation $\LL$ which is translation-invariant.) In fact, if $A$ is any $1 \times k$ irredundant integer matrix, then for all $\eps>0$ there exists an $n_0>0$ such that for all integers $n \geq n_0$ we have
$$ \mu^{*}(n,\LL,r) \leq \mu(n,\LL,r) \leq \mu^{*}(n,\LL,r)+\eps n. $$ \COMMENT{AT: actually not sure I have a proof for $\ell \times k$ matrices. But who cares since we only consider $1\times k$.}
This follows from e.g. \cite[Theorem 2]{ksv2}, since such $\LL$ have $o(n^{k-\ell})$ non-$k$-distinct solutions in $[n]$ (i.e. a solution $(x_1,\dots,x_k)$ where there is an $i \not=j$ such that $x_i=x_j$).

Consequently it is equally interesting to study $\mu^{*}(n,\LL,r)$ in the case when $\mu(n,\LL,r) =\Omega (n)$. In the case of sum-free sets (where $\LL$ is $x+y=z$), the study of $\mu^{*}(n,\LL,r)$ is a classical problem of Abbott and Wang~\cite{aw}. (Note that the only difference between $\mu(n,\LL,r)$ and $\mu^{*}(n,\LL,r)$ in this case is that $\mu(n,\LL,r)$ allows non-distinct sums $x+x=z$ whereas $\mu^{*}(n,\LL,r)$ does not.) Let $\mu(n,r):=\mu^{*}(n,\LL,r)$ where $\LL$ is $x+y=z$. An easy proof shows that $\mu(n,1)=\ceil{n/2}$. 

The following definitions help motivate the study of $\mu(n,r)$ for $r \geq 2$. Let $f(r)$ denote the largest positive integer $m$ for which there exists a partition of $[m]$ into $r$ sum-free sets, and let $h(r)$ denote the largest positive integer $m$ for which there exists a partition of $[m]$ into $r$ sets which are sum-free modulo $m+1$. 

Abbott and Wang~\cite{aw} conjectured that $h(r)=f(r)$, and showed that $\mu(n,r) \geq n-\floor{n/(h(r)+1)}$. They also proved the following upper bound.

\begin{thm}[\cite{aw}]
We have $\mu(n,r) \leq n- \floor{cn/((f(r)+1)\log(f(r)+1))}$ where $c:=e^{-\gamma} \approx 0.56$ ($\gamma$ denotes the Euler-Mascheroni constant).
\end{thm}
\COMMENT{RH: Note that $f(r)+1 \leq R_r(3,\dots,3)-1 \leq \floor{r!e}$ so this gives an upper bound of $n-\floor{cn/(r!e \log r!e)}$; our bound is approximately $n-n/(r!e)$ so is better. Better bounds on Schur numbers $f(r)$ or Ramsey numbers $R_r(3,\dots,3)$ may make the Abbott-Wang bound better than ours.}

We provide an alternate upper bound, which is a  modification of Hu's~\cite{hu} proof that $\mu(n,2) = n-\floor{\frac{n}{5}}$. (To see why this is a lower bound, consider the set $\{ x \in [n]: x \equiv 1 \text{ or } 4 \! \mod 5\} \cup \{ y \in [n]: y \equiv 2 \text{ or } 3 \! \mod 5\}$.) First we need the following fact.
Given $x \in [n]$ and $T \subseteq [n]$, write $x+T := \lbrace x+y : y \in T\rbrace$.
Given $S,T \subseteq [n]$, say that $T$ is a \emph{difference set of $S$} if there exists $x \in S$ such that $x+T \subseteq S$.

\begin{fact}\label{sumfact}
Let $n \in \mathbb{N}$ and $S,T,T' \subseteq [n]$.
\begin{itemize}
\item[(i)] If $T$ is a difference set of a sum-free set $S$, then $S \cap T = \emptyset$.
\item[(ii)] If $T'$ is a difference set of $T$, and $T$ is a difference set of $S$, then $T'$ is a difference set of $S$.
\end{itemize}
\end{fact}

\begin{proof}
If there exists $x \in S$ such that $x+T \subseteq S$ and moreover there exists $y \in S \cap T$, then $x+y \in S$, proving (i).
For (ii), suppose that there is $x' \in T$ and $x \in S$ such that $x'+T' \subseteq T$ and $x+T \subseteq S$.
Then $x+x'+T' \subseteq S$ and $x+x' \in x+T \subseteq S$, proving (ii). 
\end{proof}

\begin{thm}\label{hunew}
We have $\mu(n,r) \leq n-\floor{\frac{n}{\floor{r!e}}}$.
\end{thm}
Note that Theorem~\ref{hunew} does indeed recover Hu's bound~\cite{hu} for the case $r=2$.
\begin{proof}
Fix $n,r \in \mathbb N$.
Let $\ell(0) := 1$.
For all integers $i \geq 1$, define
$$
\ell(i):=i!\left(1+\sum_{t \in [i]}\frac{1}{t!}\right) = \floor{i!e}.
$$
\COMMENT{Proof of $i!\left(1+\sum_{t \in [i]}\frac{1}{t!}\right) = \floor{i!e}$: We have 
$$ \floor{i!e} = \left \lfloor i! \left(1+\frac{1}{1}+\frac{1}{2!} + \dots \right) \right \rfloor = \left \lfloor i! \left(1+\sum_{t \in [i]}\frac{1}{t!} + \sum_{t=i+1}^{\infty} \frac{1}{t!} \right) \right \rfloor =i!\left(1+\sum_{t \in [i]}\frac{1}{t!}\right) + \left \lfloor \sum_{t=i+1}^{\infty} \frac{i!}{t!} \right \rfloor,$$ 
and since we have $i!(i+k+1)! \geq (i+1)!(i+k)!$ for any positive integers $i$ and $k$, we have 
$$\sum_{t=i+1}^{\infty} \frac{i!}{t!} \leq \sum_{t=2}^{\infty} \frac{1}{t!} = e-2<1,$$ hence the result follows.
\\
AT: I see it as that it sufficient to show, for all $k=1,2...$,
$$i!/(i+k)!\leq 1/(1+k)!.$$
This follows as $i!\leq (2+k)\dots (i+k)$.}
Note that $\ell(i)=i \ell(i-1)+1$ for all $i \geq 1$.
Choose the unique $q \in \mathbb{N} \cup \{0\}$ and $0 \leq k \leq \ell(r)-1$ such that $n=\ell(r)q+k$.
Consider any partition 
  $S_1 \dot{\cup} \cdots \dot{\cup} S_r \dot{\cup} R=[n]$, where each $S_i$ is sum-free. 
We wish to show that $|R| \geq q$, since then $\mu(\ell(r)q+k,r) \leq (\ell(r)-1)q+k$ and so $\mu(n,r) \leq n-\floor{n/\ell(r)}$.

Suppose not.
We will obtain integers $\lbrace j_1,\ldots,j_r\rbrace = [r]$ and subsets $D_0, D_1, \ldots, D_r$ of $[n]$ such that the following properties hold for all $0 \leq i \leq r$.
\begin{itemize}
\item[$P_1(i)$] $|D_i| \geq \ell(r-i)q$;
\item[$P_2(i)$] $D_i$ is a difference set of $S_{j_t}$ for all $t \in [i]$;
\item[$P_3(i)$] $D_i \cap S_{j_t} = \emptyset$ for all $t \in [i]$.
\end{itemize}
Let $D_0 := [n]$.
Then $P_1(0)$ holds by definition, and $P_2(0)$ and $P_3(0)$ are vacuous.
Suppose, for some $0 \leq i < r$, we have obtained distinct $\lbrace j_1,\ldots,j_{i}\rbrace \subseteq [r]$ and $D_0,D_1,\ldots,D_{i}$ such that $P_1(t)$--$P_3(t)$ hold for all $t \in [i]$.

Suppose that $|D_{i} \cap \bigcup_{t \in [r] \setminus \{j_1,\dots,j_i\}}S_{t}| \leq (\ell(r-i)-1)q$.
Then we have that
$$
|D_i \cap R| \stackrel{P_3(i)}{\geq} |D_i|-(\ell(r-i)-1)q \stackrel{P_1(i)}{\geq} q,
$$
a contradiction. So
by averaging, there exists $j_{i+1} \in [r]\setminus \lbrace j_1,\ldots,j_i\rbrace$ such that
$$
|D_i \cap S_{j_{i+1}}| \geq \left\lceil \frac{(\ell(r-i)-1)q+1}{r-i} \right\rceil = \ell(r-i-1)q+1.
$$
Thus we can write $D_i \cap S_{j_{i+1}} \supseteq \lbrace s_{i,0} < \ldots < s_{i,\ell(r-i-1)q} \rbrace$.
Let $D_{i+1} := \lbrace s_{i,x}-s_{i,0} : x \in [\ell(r-i-1)q]\rbrace$.
We claim that $P_1(i+1)$--$P_3(i+1)$ hold.
Property $P_1(i+1)$ is clear by definition.
For $P_2(i+1)$, note that $D_{i+1}$ is a difference set of both $D_i$ and $S_{j_{i+1}}$.
Then Fact~\ref{sumfact}(ii) and $P_2(i)$ imply that additionally $D_{i+1}$ is a difference set of $S_{j_t}$ for all $t \in [i]$.
Fact~\ref{sumfact}(i) implies that $D_{i+1} \cap S_{j_{t}} = \emptyset$ for all $t \in [i+1]$, proving $P_3(i+1)$.

Thus we  obtain $D_r$ satisfying $P_1(r)$--$P_3(r)$.
By $P_1(r)$ and $P_3(r)$ we have that $|D_r| \geq \ell(0)q=q$ and $D_r \subseteq R$, a contradiction. 
\end{proof}

\subsection{Open Problem}
We conclude the section with an open problem.
Recall  Hu~\cite{hu} showed that $\mu(n,2) = n-\floor{\frac{n}{5}}$. So in the case when $\LL$ is $x+y=z$, Theorem~\ref{radonum} implies that there are $2^{4n/5+o(n)}$ $(\LL,2)$-free subsets of $[n]$. 
We believe the error term in the exponent here can be replaced by a constant.
\begin{conj}\label{conjCE}
Let $\LL$ denote $x+y=z$. There are $\Theta (2^{4n/5})$ $(\LL,2)$-free subsets of $[n]$.
\end{conj}
Note that Conjecture~\ref{conjCE} can be viewed as a $2$-coloured analogue of the Cameron--Erd\H{o}s conjecture~\cite{cam1} which was famously resolved by Green~\cite{G-CE} and independently  Sapozhenko~\cite{sap}.

Since our paper was submitted, Tran~\cite{tran} has proved a slight variant of Conjecture~\ref{conjCE}; that is, he proves the result where one instead defines sum-free to also forbid non-distinct sums $x+x=z$ (as in the previous section). Note Tran's result does not quite imply Conjecture~\ref{conjCE} directly.

\section{Applications of the container method to graph Ramsey theory}\label{sec5}
%
In this section we answer some questions in hypergraph Ramsey theory, introduced in Sections~\ref{ressec} and~\ref{countsec}. How many $n$-vertex hypergraphs are not Ramsey, and what does a typical such hypergraph look like? How dense must the Erd\H{o}s-R\'enyi random hypergraph be to have the Ramsey property with high probability, and above this threshold, how strongly does it possess the Ramsey property?

Our main results here are applications of the asymmetric container theorem (Proposition~\ref{asymmulti}).
For arbitrary $k$-uniform hypergraphs $H_1,\ldots,H_r$, we first prove Theorem~\ref{ramseycont}, a container theorem for non-$(H_1,\ldots,H_r)$-Ramsey $k$-uniform hypergraphs.
To see how one might prove such a theorem, observe that, if $\mathcal{H}_i$ is the hypergraph of copies of $H_i$ on $n$ vertices (i.e.~vertices correspond to $k$-subsets of $[n]$, and edges correspond to copies of $E(H_i)$; see Definition~\ref{copiesgraph}), then every non-$(H_1,\ldots,H_r)$-Ramsey $k$-uniform hypergraph $G$ corresponds to a set in $\mathcal{I}(\mathcal{H}_1,\ldots,\mathcal{H}_r)$.
We then use Theorem~\ref{ramseycont} to:
\begin{itemize}
\item[(1)] count the number of $k$-uniform hypergraphs on $n$ vertices which are not $(H_1,\ldots,H_r)$-Ramsey (Theorem~\ref{ramseycount});
\item[(2)] determine the global resilience of $G^{(k)}_{n,p}$ with respect to the property of being $(H_1,\ldots,H_r)$-Ramsey (Theorem~\ref{ramres}).
That is, we show that there is a constant $C$ such that whenever $p \geq Cn^{-1/m_k(H_1)}$, we obtain a function $t$ of $n$ and $p$ such that, with high probability, any subhypergraph $G \subseteq G^{(k)}_{n,p}$ with $e(G)>t+\Omega(pn^k)$ is $(H_1,\ldots,H_r)$-Ramsey. Further, there is some $G' \subseteq G^{(k)}_{n,p}$ with $e(G') > t-o(pn^k)$ which is \emph{not} $(H_1,\ldots,H_r)$-Ramsey.
\item[(3)] As a corollary of (2), we see that, whenever $p \geq Cn^{-1/m_k(H_1)}$, the random hypergraph $G^{(k)}_{n,p}$ is $(H_1,\ldots,H_r)$-Ramsey with high probability.
\end{itemize}

The statements of (1)--(3) all involve a common parameter: the maximum size $\ex^r(n;H_1,\ldots,H_r)$ of an $n$-vertex $k$-uniform hypergraph which is not $(H_1,\ldots,H_r)$-Ramsey.
For this reason, we generalise the classical supersaturation result of Erd\H{o}s and Simonovits~\cite{supersat} to show that any $n$-vertex $k$-uniform hypergraph $G$ with at least $\ex^r(n;H_1,\ldots,H_r)+\Omega(n^k)$ edges is somehow `strongly' $(H_1,\ldots,H_r)$-Ramsey.
In the graph case, an old result of Burr, Erd\H{o}s and Lov\'asz~\cite{bes} allows us to quite accurately determine $\ex^r(n;H_1,\ldots,H_r)$.

\subsection{Definitions and notation}

In this section, $k \geq 2$ is an integer and we use \emph{$k$-graph} as shorthand for $k$-uniform hypergraph.
Recall from Section~\ref{ressec} that, given $r \in \mathbb{N}$ and a $k$-graph $G$, an \emph{$r$-colouring} is a function $\sigma : E(G) \rightarrow [r]$.
Given $k$-graphs $H_1,\ldots,H_r$, we say that $\sigma$ is \emph{$(H_1,\ldots,H_r)$-free} if $\sigma^{-1}(i)$ is $H_i$-free for all $i \in [r]$.
Then $G$ is \emph{$(H_1,\ldots,H_r)$-Ramsey} if it has no $(H_1,\ldots,H_r)$-free $r$-colouring.

Given an integer $\ell \geq k$, denote by $K^{(k)}_\ell$ the complete $k$-graph on $\ell$ vertices.
A $k$-graph $H$ is \emph{$k$-partite} if the vertices of $H$ can be $k$-coloured so that each edge contains one vertex of each colour.
Given a $k$-graph $S$, recall the definitions
$$
d_k(S) := \begin{cases}
0 &\text{ if } e(S)=0; \\
1/k &\text{ if } v(S)=k \text{ and } e(S)=1;\\
\frac{e(S)-1}{v(S)-k} &\text{ otherwise}
\end{cases}
$$
and $$
m_k(S) := \max_{S' \subseteq S}d_k(S').$$
\subsection{The maximum density of a hypergraph which is not Ramsey}\label{maxsec}

Given integers $n \geq k$ and a $k$-graph $H$, we denote by $\ex(n;H)$ the maximum size of an $n$-vertex $H$-free $k$-graph. Define the \emph{Tur\'an density} $\pi(H)$ of $H$ by
\begin{equation}
\pi(H) := \lim_{n \rightarrow \infty}\frac{\ex(n;H)}{\binom{n}{k}}
\end{equation}
(which exists by a simple averaging argument, see~\cite{kns}).
The so-called supersaturation phenomenon discovered by Erd\H{o}s and Simonovits~\cite{supersat} asserts that any sufficiently large hypergraph with density greater than $\pi(H)$ contains not just one copy of $H$, but in fact a positive fraction of $v(H)$-sized sets span a copy of $H$. 
Note supersaturation problems date back to a result of Rademacher (see~\cite{erdrad}).

\begin{theorem}[\cite{supersat}]\label{ess}
For all $k \in \mathbb{N}$; $\delta > 0$ and all $k$-graphs $H$, there exist $n_0,\eps>0$ such that for all integers $n \geq n_0$,
every $n$-vertex $k$-graph $G$ with $e(G) \geq \left(\pi(H)+\delta\right)\binom{n}{k}$ contains at least $\eps\binom{n}{v(H)}$ copies of $H$.
\end{theorem}

When $k=2$, the Erd\H{o}s--Stone--Simonovits theorem~\cite{erdosstone} says that for all graphs $H$, the value of $\pi(H)$ is determined by the chromatic number $\chi(H)$ of $H$, via
\begin{equation}\label{pi}
\pi(H) = 1-\frac{1}{\chi(H)-1}.
\end{equation}
For $k \geq 3$, the value of $\pi(H)$ is only known for a small family of $k$-graphs $H$.
It remains an open problem to even determine the Tur\'an density of $K^{(3)}_4$, the smallest non-trivial complete $3$-graph (the widely-believed conjectured value is $\frac{5}{9}$). 
For more background on this, the so-called \emph{hypergraph Tur\'an problem}, the interested reader should consult the excellent survey of Keevash~\cite{keevash}.

In this section, we generalise Theorem~\ref{ess} from $H$-free hypergraphs to non-$(H_1,\ldots,H_r)$-Ramsey hypergraphs (note that a hypergraph is $H$-free if and only if it is \emph{not} $(H)$-Ramsey).
Given $\eps>0$, we say that an $n$-vertex $k$-graph $G$ is \emph{$\eps$-weakly $(H_1,\ldots,H_r)$-Ramsey}
if there exists an $r$-colouring $\sigma$ of $G$ such that, for all $i \in [r]$, the number of copies of $H_i$ in $\sigma^{-1}(i)$ is less than $\eps\binom{n}{v(H_i)}$.
Otherwise, $G$ is \emph{$\eps$-strongly $(H_1,\ldots,H_r)$-Ramsey}. Note that $\eps$-weakly $(H_1,\ldots,H_r)$-Ramsey graphs may not in fact be $(H_1,\ldots,H_r)$-Ramsey.


Using a well-known averaging argument of Katona, Nemetz and Simonovits~\cite{kns}, we can show that $\binom{n}{k}^{-1}\ex^r(n;H_1,\ldots,H_r)$ converges as $n$ tends to infinity.
Indeed, let $G$ be an $n$-vertex non-$(H_1,\ldots,H_r)$-Ramsey graph with $e(G) = \ex^r(n;H_1,\ldots,H_r)$.
The average density of an $(n-1)$-vertex induced subgraph of $G$ is precisely
$$
\binom{n}{n-1}^{-1} \sum_{U \subseteq V(G): |U|=n-1}\frac{e(G[U])}{\binom{n-1}{k}} = (n-k)^{-1}\cdot \binom{n}{k}^{-1}\sum_{U \subseteq V(G):|U|=n-1}e(G[U]) = \binom{n}{k}^{-1}e(G).
$$
But the left-hand side is at most $\binom{n-1}{k}^{-1}\cdot\ex^r(n-1;H_1,\ldots,H_r)$, otherwise $G$ would contain an $(n-1)$-vertex subgraph which is $(H_1,\ldots,H_r)$-Ramsey, violating the choice of $G$.
We have shown that
$$
\frac{\ex^r(n;H_1,\ldots,H_r)}{\binom{n}{k}}
$$
is a non-increasing function of $n$ (which is bounded below, by $0$), and so this function has a limit. Therefore we may define the \emph{$r$-coloured Tur\'an density} $\pi(H_1,\ldots,H_r)$ of $(H_1,\ldots,H_r)$ by
$$
\pi(H_1,\ldots,H_r) := \lim_{n \rightarrow \infty} \frac{\ex^r(n;H_1,\ldots,H_r)}{\binom{n}{k}}.
$$


As for $k \geq 3$, the problem of determining $\pi(H)$ is still out of reach, we certainly cannot evaluate $\pi(H_1,\ldots,H_r)$ in general.
However, any non-$(H_1,\ldots,H_r)$-Ramsey graph is $K^{(k)}_s$-free, where $s := R(H_1,\ldots,H_r)$ is the smallest integer $m$ such that $K^{(k)}_m$ is $(H_1,\ldots,H_r)$-Ramsey.
Thus
\begin{equation}\label{pibound}
\pi(H_1,\ldots,H_r) \leq \pi(K_s^{(k)}),
\end{equation}%
\COMMENT{lower bound?}
which is at most $1 - \binom{s-1}{k-1}^{-1}$ (de Caen~\cite{decaen}).
An interesting question is for which $H_1,\ldots,H_r$ the inequality in~(\ref{pibound}) is tight.
We discuss the case $k=2$ in detail in Section~\ref{graphsec}.

We now state the main result of this subsection, which generalises Theorem~\ref{ess} to $r \geq 1$. The proof follows a standard approach to proving supersaturation results.

\begin{theorem}\label{epsclose}
For all $\delta > 0$, integers $r \geq 1$ and $k \geq 2$, and $k$-graphs $H_1,\ldots,H_r$, there exist $n_0,\eps > 0$ such that for all integers $n \geq n_0$,
every $n$-vertex $k$-graph $G$ with 
$e(G) \geq \left(\pi(H_1,\ldots,H_r)+\delta\right)\binom{n}{k}$ is $\eps$-strongly $(H_1,\ldots,H_r)$-Ramsey.
\end{theorem}

\begin{proof}
Let $\delta > 0$ and let $r,k$ be positive integers with $k \geq 2$.
By the definition of $\pi(\cdot)$, there exists $m_0>0$ such that for all integers $m \geq m_0$,
$$
\ex^r(m;H_1,\ldots,H_r)<\left(\pi(H_1,\ldots,H_r)+\frac{\delta}{2}\right)\binom{m}{k}.
$$
Fix an integer $m \geq m_0$.
Without loss of generality, we may assume that $m \geq v(H_i)$ for all $i \in [r]$.
Choose $\eps>0$ to be such that
$$
\eps \leq \frac{\delta}{2r}\binom{m}{v(H_i)}^{-1}
$$
for all $i \in [r]$.
Let $n$ be an integer which is sufficiently large compared to $m$, and let $G$ be a $k$-graph on $n$ vertices with $e(G) = (\pi(H_1,\ldots,H_r)+\delta) \binom{n}{k}$.
We need to show that, for every $r$-colouring $\sigma$ of $G$, there is $i \in [r]$ such that $\sigma^{-1}(i)$ contains at least $\eps\binom{n}{v(H_i)}$ copies of $H_i$; so fix an arbitrary $\sigma$.

Define $\mathcal{M}$ to be the set of $M \in \binom{V(G)}{m}$ such that $e(G[M]) \geq (\pi(H_1,\ldots,H_r)+\frac{\delta}{2})\binom{m}{k}$.
Then
\begin{align*}
\sum_{U \subseteq V(G): |U|=m}e(G[U]) \leq |\mathcal{M}|\binom{m}{k} + \left(\binom{n}{m}-|\mathcal{M}|\right)\left(\pi(H_1,\ldots,H_r)+\frac{\delta}{2}\right)\binom{m}{k}.
\end{align*}
But for every $e \in E(G)$, there are exactly $\binom{n-k}{m-k}$ sets $U \subseteq V(G)$ with $|U|=m$ such that $e \in E(G[U])$. Thus also
\begin{align*}
\sum_{U \subseteq V(G): |U|=m}e(G[U]) \geq \binom{n-k}{m-k}(\pi(H_1,\ldots,H_r)+\delta)\binom{n}{k} = (\pi(H_1,\ldots,H_r)+\delta)\binom{n}{m}\binom{m}{k},
\end{align*}
and so, rearranging, we have $|\mathcal{M}| \geq \delta\binom{n}{m}/2$.
By the choice of $m$, for every $M\in \mathcal{M}$, there exists $i=i(M) \in [r]$ such that $\sigma^{-1}(i)$ contains a copy of $H_i$ with vertices in $M$.
Choose $\mathcal{M}' \subseteq \mathcal{M}$ such that the $i(M')$ are equal for all $M'\in \mathcal{M}'$ and $|\mathcal{M}'| \geq |\mathcal{M}|/r$.
Without loss of generality let us assume that $i(M')=1$ for all $M' \in \mathcal{M}'$.
So for each $M' \in \mathcal{M}'$, there is a  copy of $H_1\subseteq G[ M']$ which is monochromatic with colour $1$ under $\sigma$.
Each such copy has vertex set contained in at most $\binom{n-v(H_1)}{m-v(H_1)}$ sets $M' \in \mathcal{M}'$.
Thus the number of such monochromatic copies of $H_1$ in $G$ is at least
$$
\frac{\frac{\delta}{2}\cdot  \binom{n}{m}}{r\binom{n-v(H_1)}{m-v(H_1)}} = \frac{\delta}{2r}\cdot \binom{m}{v(H_1)}^{-1} \cdot \binom{n}{v(H_1)} \geq \eps\binom{n}{v(H_1)}.
$$
So $G$ is $\eps$-strongly $(H_1,\ldots,H_r)$-Ramsey, as required.
\end{proof}

\subsection{The special case of graphs: maximum size and typical structure}\label{graphsec}

The intimate connection between forbidden subgraphs and chromatic number when $k=2$ allows us to make some further remarks here.
(This section is separate from the remainder of the paper and the results stated here will not be required later on.)

\subsubsection{The maximum number of edges in a graph which is not Ramsey}

Given $s,n \in \mathbb{N}$, let $T_s(n)$ denote the $s$-partite Tur\'an ($2$-)graph on $n$ vertices; that is, the vertex set of $T_s(n)$ has a partition into $s$ parts $V_1,\ldots,V_s$ such that $\left| |V_i|-|V_j|\right| \leq 1$ for all $i,j \in [s]$; and $xy$ is an edge of $T_s(n)$ if and only if there are $ij \in \binom{[s]}{2}$ such that $x \in V_i$ and $y \in V_j$.
Write $t_s(n) := e(T_s(n))$.

We need to define two notions of Ramsey number.

\begin{definition}[Ramsey number, chromatic Ramsey number and chromatic Ramsey equivalence]
Given an integer $r \geq 1$ and families $\mathcal{H}_1,\ldots,\mathcal{H}_r$ of graphs, the \emph{Ramsey number $R(\mathcal{H}_1,\ldots,\mathcal{H}_r)$} is the least $m$ such that any $r$-colouring of $K_m$ contains an $i$-coloured copy of $H_j$ for some $i \in [r]$ and  some $H_j \in \mathcal{H}_i$.
If $\mathcal{H}_i = \lbrace K_{\ell_i} \rbrace$ for all $i \in [r]$ then we instead write $R(\ell_1,\ldots,\ell_r)$, and simply $R^r(\ell)$ in the case when $\ell_1=\ldots=\ell_r=:\ell$.

Given graphs $H_1,\ldots,H_r$, the \emph{chromatic Ramsey number $R_\chi(H_1,\ldots,H_r)$} is the least $m$ for which there exists an $(H_1,\ldots,H_r)$-Ramsey graph with chromatic number $m$.

\end{definition}

Trivially, for any $k$-graph $H$, we have that $R_\chi(H) = \chi(H)$.
If $H_1,\ldots,H_r$ are graphs, then
\begin{equation}\label{graphepsclose}
t_{R_\chi(H_1,\ldots,H_r)-1}(n) \leq \ex^r(n;H_1,\ldots,H_r) \leq t_{R_\chi(H_1,\ldots,H_r)-1}(n) + o(n^2).
\end{equation}
Thus
\begin{equation}\label{piR}
\pi(H_1,\ldots,H_r) = 1 - \frac{1}{R_\chi(H_1,\ldots,H_r)-1} = \pi\left(K_{R_\chi(H_1,\ldots,H_r)}\right).
\end{equation}
The first inequality in 
(\ref{graphepsclose}) follows by definition of  $\ex^r(n;H_1,\ldots,H_r)$; the second  from~(\ref{pi}) applied with a graph $H$ which is $(H_1,\ldots,H_r)$-Ramsey and has $\chi(H) = R_\chi(H_1,\ldots,H_r)$.
Clearly, then, $\pi(H_1,\ldots,H_r)=\pi(J_1,\ldots,J_r)$ if and only if $R_\chi(J_1,\ldots,J_r)=R_\chi(H_1,\ldots,H_r)$.
So, in the graph case, the inequality~(\ref{pibound}) is tight when the Ramsey number and chromatic Ramsey number coincide.

As noted by Bialostocki, Caro and Roditty~\cite{bcr}, one can determine $\ex^r(n;H_1,\ldots,H_r)$ exactly in the case when $H_1,\ldots,H_r$ are cliques of equal size.

\begin{theorem}[\cite{bcr}]\label{bcrthm}
For all positive integers $\ell,n \geq 3$ and $r \geq 1$, we have $\ex^r(n;K_\ell,\ldots,K_\ell)=t_{R^r(\ell)-1}(n)$.
\end{theorem}

Thus in this case~(\ref{pibound}) is tight.
The chromatic Ramsey number was introduced by Burr, Erd\H{o}s and Lov\'asz~\cite{bes} who showed that, in principle, one can determine $R_\chi$ given the usual Ramsey number $R$.
A \emph{graph homomorphism} from a graph $H$ to a graph $K$ is a function $\phi : V(H) \rightarrow V(K)$ such that $\phi(x)\phi(y) \in E(K)$ whenever $xy \in E(H)$.
Let $\Hom(H)$ denote the set of all graphs $K$ such that there exists a graph homomorphism $\phi$ for which $K=\phi(H)$.
Since there exists a homomorphism from $H$ into $K_\ell$ if and only if $\chi(H) \leq \ell$, we also have that $R(\Hom(H))=\chi(H)$. Thus $R(\Hom(H)) = R_\chi(H)$.
In fact this relationship extends to all $r \geq 1$.

\begin{lemma}[\cite{bes,chvatal,lin}]\label{chromram}
For all integers $r \in \mathbb{N}$ and graphs $H_1,\ldots,H_r$,
$$
R_\chi(H_1,\ldots,H_r) = R(\mathrm{Hom}(H_1),\ldots,\mathrm{Hom}(H_r)).
$$
Moreover, for all integers $\ell_1,\ldots,\ell_r \geq 3$, we have that
$$
R_\chi(K_{\ell_1},\ldots,K_{\ell_r}) = R(\ell_1,\ldots,\ell_r).
$$
\end{lemma}
\COMMENT{is the lemma implied by (\ref{graphepsclose}) and Lemma~\ref{pilem}??}

The second statement is a corollary of the first since $\mathrm{Hom}(K_\ell) = \lbrace K_\ell \rbrace$.%
\COMMENT{this is an equals and not a $\supseteq$ because any $K$ in $\Hom(K_\ell)$ must be the \emph{image} of $\phi$.}
Another observation (see~\cite{bes}) is that for all $\ell \in \mathbb{N}$, the chromatic Ramsey number
$R_\chi(C_{2\ell+1},C_{2\ell+1})$ is equal to $5$ if $\ell=2$, and equal to $6$ otherwise.

The first inequality in~(\ref{graphepsclose}) is not always tight, for example when $H$ is the disjoint union of two copies of some graph $G$.
Indeed, $\Hom(H) \supseteq \Hom(G)$ and so $R_\chi(H,\ldots,H) = R_\chi(G,\ldots,G)$.
Let $F$ be an $n$-vertex graph with $e(F) = \ex^r(n;G,\ldots,G)$ which is not $(G,r)$-Ramsey.
Obtain a graph $T$ by adding an edge $e$ to $F$.
Then there exists an $r$-colouring of $T$ in which every monochromatic copy of $G$ contains $e$ (the monochromatic-$G$-free colouring of $F$, with $e$ arbitrarily coloured).
Hence $T$ is not $(H,r)$-Ramsey and so
$$
\ex^r(n;H,\ldots,H) > \ex^r(n;G,\ldots,G) \geq t_{R_\chi(G,\ldots,G)}(n) = t_{R_\chi(H,\ldots,H)}(n).
$$

We say that a graph $H$ is \emph{(weakly) colour-critical} if there exists $e \in E(H)$ for which $\chi(H-e) < \chi(H)$.
Complete graphs and odd cycles are examples of colour-critical graphs.
The following conjecture would generalise Theorem~\ref{bcrthm} to provide a large class of graphs where the first inequality in~(\ref{graphepsclose}) is tight.

\begin{conjecture}\label{conj1}
Let $r$ be a positive integer and $H$ a colour-critical graph.
Then, whenever $n$ is sufficiently large,
$$
\ex^r(n;H,\ldots,H) = t_{R_\chi(H,\ldots,H)-1}(n).
$$
\end{conjecture}%
\COMMENT{KS: This is more of a reminder for us to think about this conjecture! I feel like the proof should just be a standard application of the stability method.}
If true, this conjecture would also generalise a well-known result of Simonovits~\cite{simonovits} which extends Tur\'an's theorem to colour-critical graphs.
It would also determine $\ex^r(n;H,\ldots,H)$ explicitly whenever $H$ is an odd cycle.

\subsubsection{The typical structure of non-Ramsey graphs}\label{sec:struct}
There has been much interest in determining the \emph{typical} structure of an $H$-free graph. For example, Kolaitis, Pr\"omel and Rothschild~\cite{kpr} proved that almost all $K_r$-free graphs are $(r-1)$-partite.
It turns out that one can easily obtain a result on the typical structure of non-Ramsey graphs from a result of Pr\"omel and Steger~\cite{promelsteger}.

Given two families $\mathcal{A}(n),\mathcal{B}(n)$ of $n$-vertex graphs such that $\mathcal{B}(n)\subseteq\mathcal{A}(n)$, we say that \emph{almost all} $n$-vertex graphs $G \in \mathcal{A}(n)$ are in $\mathcal{B}(n)$ if
$$
\lim_{n \rightarrow \infty}\frac{|\mathcal{A}(n)|}{|\mathcal{B}(n)|} = 1.
$$

The next result of  Pr\"omel and Steger~\cite{promelsteger} immediately tells us the typical structure of non-Ramsey graphs in certain cases.
\begin{theorem}[\cite{promelsteger}]\label{ps}
For every graph $H$, the following holds.
Almost all $H$-free graphs are $(\chi(H)-1)$-partite if and only if $H$ is colour-critical.
\end{theorem}

\begin{corollary}\label{structure2}
For all integers $r$ and graphs $H_1,\ldots,H_r$, if there exists an $(H_1,\ldots,H_r)$-Ramsey graph $H$ such that $\chi(H) = R_\chi(H_1,\ldots,H_r)$ and $H$ is colour-critical, then 
almost every non-$(H_1,\ldots,H_r)$-Ramsey graph is $(R_\chi(H_1,\ldots,H_r)-1)$-partite.%
\COMMENT{other direction? conclusion not true if there is no such $H$?}
\end{corollary}
\proof
The result follows since every non-$(H_1,\ldots,H_r)$-Ramsey graph $G$ is $H$-free, and every \\ $(R_\chi(H_1,\ldots,H_r)-1)$-partite graph is non-$(H_1,\ldots,H_r)$-Ramsey.
\endproof

In particular, if in Corollary~\ref{structure2}, each $H_i$ is a clique, say $H_i = K_{\ell_i}$, then by Lemma~\ref{chromram} we can take $H := K_{R(\ell_1,\ldots,\ell_r)}$.
So, for example, almost every non-$(K_3,2)$-Ramsey graph is $5$-partite.

\subsection{A container theorem for Ramsey hypergraphs}

Recall that $\overline{\mathrm{Ram}}(n;H_1,\ldots,H_r)$ is the set of $n$-vertex $k$-graphs which are not $(H_1,\ldots,H_r)$-Ramsey and $\mathrm{Ram}(H_1,\ldots,H_r)$ is the set of $(H_1,\ldots,H_r)$-Ramsey $k$-graphs (on any number of vertices).
Recall further that an $H$-free $k$-graph is precisely a non-$(H,1)$-Ramsey graph.
Write $\mathcal{G}_k(n)$ for the set of all $k$-graphs on vertex set $[n]$.
Let $\mathcal{I}_r(n;H_1,\ldots,H_r)$ denote the set of all ordered $r$-tuples $(G_1,\ldots,G_r) \in (\mathcal{G}_k(n))^r$ of $k$-graphs such that each $G_i$ is $H_i$-free and $E(G_i) \cap E(G_j) = \emptyset$ for all distinct $i,j \in [r]$.
Note that for any $G \in \overline{\mathrm{Ram}}(n;H_1,\ldots,H_r)$, there exist pairwise edge-disjoint $k$-graphs $G_1,\ldots,G_r$ such that $\bigcup_{i \in [r]}G_i = G$ and $(G_1,\ldots,G_r) \in \mathcal{I}_r(n;H_1,\ldots,H_r)$.
In this subsection, we prove a container theorem for elements in $\mathcal{I}_r(n;H_1,\ldots,H_r)$.
To do so, we will apply Proposition~\ref{asymmulti} to hypergraphs $\mathcal{H}_1,\ldots,\mathcal{H}_r$, where $\mathcal{H}_i$ is the hypergraph of copies of $H_i$ (see Definition~\ref{copiesgraph}).
In $\mathcal{H}_i$, an independent set corresponds to an $H_i$-free $k$-graph.

\begin{definition}\label{copiesgraph}
Given an integer $k \geq 2$, a $k$-graph $H$ and positive integer $n$, the \emph{hypergraph $\mathcal{H}$ of copies of $H$ in $K^{(k)}_n$} has vertex set $V(\mathcal{H}) := \binom{[n]}{k}$, and $E \subseteq \binom{V(\mathcal{H})}{e(H)}$ is an edge of $\mathcal{H}$ if and only if $E$ is isomorphic to $E(H)$. 
\end{definition}

We will need the following simple proposition from~\cite{container1}.

\begin{proposition}[\cite{container1}, Proposition~7.3]\label{copieshyp}
Let $H$ be a $k$-graph.
Then there exists $c>0$ such that, for all positive integers $n$, the following holds.
Let $\mathcal{H}$ be the $e(H)$-uniform hypergraph of copies of $H$ in $K^{(k)}_n$.
Then, letting $p = n^{-1/m_k(H)}$,
$$
\Delta_\ell(\mathcal{H}) \leq c \cdot p^{\ell-1}\frac{e(\mathcal{H})}{v(\mathcal{H})},
$$
for every $\ell \in [e(H)]$.
\end{proposition}

We can now prove our container theorem for elements in $\mathcal{I}_r(n;H_1,\ldots,H_r)$.

\begin{theorem}\label{ramseycont}
Let $r,k \in \mathbb{N}$ with $k \geq 2$ and $\delta > 0$.
Let $H_1,\ldots,H_r$ be $k$-graphs such that $m_k(H_1) \geq \ldots \geq m_k(H_r)$ and $\Delta_1(H_i) \geq 2$ for all $i \in [r]$.
Then there exists $D > 0$ such that the following holds.
For all $n \in \mathbb{N}$, there is a collection $\mathcal{S}_r \subseteq (\mathcal{G}_k(n))^r$ and a function $f:\mathcal{S}_r \rightarrow (\mathcal{G}_k(n))^r$ such that:
\begin{itemize}
\item[(i)] For all $(I_1,\ldots,I_r) \in \mathcal{I}_r(n;H_1,\ldots,H_r)$, there exists $S \in \mathcal{S}_r$ such that $S \subseteq (I_1,\ldots,I_r) \subseteq f(S)$.
\item[(ii)] If $(S_1,\ldots,S_r) \in \mathcal{S}_r$ then $\sum_{i \in [r]}e(S_i) \leq Dn^{k-1/m_k(H_1)}$.
\item[(iii)] Every $S \in \mathcal{S}_r$ satisfies $S \in \mathcal{I}_r(n;H_1,\ldots,H_r)$.%
\COMMENT{Why not say $\mathcal{S}_r \subseteq \mathcal{I}_r(n;H_1,\ldots,H_r)$?}
\item[(iv)] Given any $S = (S_1,\ldots,S_r) \in \mathcal{S}_r$, write $f(S) =: (f(S_1),\ldots,f(S_r))$. Then
\begin{itemize}
\item[(a)] $\bigcup_{i \in [r]}f(S_i)$ is $\delta$-weakly $(H_1,\ldots,H_r)$-Ramsey; and
\item[(b)] $e\left(\bigcup_{i \in [r]}f(S_i)\right) \leq \ex ^r(n;H_1,\ldots,H_r)+\delta\binom{n}{k}$.
\end{itemize}
\end{itemize}
\end{theorem}

Note that if $H$ is a $k$-graph with $\Delta_1(H)=1$, then $H$ is a matching, i.e.~a set of vertex-disjoint edges.

\begin{proof}
We will identify any hypergraph which has vertex set $[n]$ with its edge set.
It suffices to prove the theorem when $n$ is sufficiently large; otherwise we can set $\mathcal{S}_r$
to be $\mathcal{I}_r(n;H_1,\ldots,H_r)$; set $f$ to be the identity function and choose $D$ to be large.
We may further assume that there are no isolated vertices in $H_i$ for any $i \in [r]$.

Apply Proposition~\ref{copieshyp} with input hypergraphs $H_1,\ldots,H_r$ to obtain $c >0$ such that its conclusion holds with $H_i$ playing the role of $H$, for all $i \in [r]$.
Let $\delta > 0$, $r \in \mathbb{N}$ and $k \geq 2$ be given and apply Theorem~\ref{epsclose} (with $\delta /2$ playing the role of $\delta$) to obtain $n_0,\eps>0$.
Without loss of generality we may assume $\eps \leq \delta < 1$. For each $i \in [r]$, let $v_i:=v(H_i)$ and 
 $m_i := e(H_i)$ for all $i \in [r]$. Set $v:=\max_{i \in [r]}v_i$; $m := \max_{i \in [r]}m_i$;
$$
\eps' := \frac{\eps}{2\cdot v!}; \quad\text{and}\quad \eps'' := \frac{\eps'}{\binom{v}{k}\cdot v!}.
$$

Apply Proposition~\ref{asymmulti} with parameters $r,m_1,\dots,m_r,c,\eps''$ playing the roles of $r,k_1,\ldots,k_r,c,\eps$ respectively to obtain $D>0$.
Increase $n_0$ if necessary so that $0 < 1/n_0 \ll 1/D,1/k,1/r,\eps,\delta$ and let $n \geq n_0$ be an integer.

Let $\mathcal{H}_{n,i}$ be the hypergraph of copies of $H_i$ in $K^{(k)}_n$.
That is, $V(\mathcal{H}_{n,i}) := \binom{[n]}{k}$ and for each $m_i$-subset $E$ of $\binom{[n]}{k}$, put $E \in E(\mathcal{H}_{n,i})$ if and only if $E$ is isomorphic to a copy of $H_i$.
By definition, $\mathcal{H}_{n,i}$ is an $m_i$-uniform hypergraph and an independent set in $\mathcal{H}_{n,i}$ corresponds to an $H_i$-free $k$-graph with vertex set $[n]$.
Since $H_i$ is a $k$-graph with no isolated vertices,
\begin{equation}\label{eHni}
e(\mathcal{H}_{n,i}) = \frac{v_i!}{|\mathrm{Aut}(H_i)|}\binom{n}{v_i}
\end{equation}
where $\mathrm{Aut}(H_i)$ is the automorphism group of $H_i$.
For all $i \in [r]$, let
$$
\mathcal{F}_{n,i} := \left\lbrace A \subseteq \binom{[n]}{k}: e(\mathcal{H}_{n,i}[A]) \geq \eps' e(\mathcal{H}_{n,i}) \right\rbrace.
$$
We claim that $\mathcal{H}_{n,1},\ldots,\mathcal{H}_{n,r}$ and $\mathcal{F}_{n,1},\ldots,\mathcal{F}_{n,r}$ satisfy the hypotheses of Proposition~\ref{asymmulti} with the parameters chosen as above and with 
$$
p = p(n) := n^{-1/m_k(H_1)}.
$$

Clearly each family $\mathcal{F}_{n,i}$ is increasing, and $\mathcal{H}_{n,i}$ is $(\mathcal{F}_{n,i},\eps')$-dense.
Next, we show that $|A| \geq \eps''\binom{n}{k}$ for all $A \in \mathcal{F}_{n,i}$.
In any $k$-graph on $n$ vertices, there are at most $v_i !\binom{n-k}{v_i-k}$ copies of $H_i$ that contain some fixed set $\lbrace x_1,\ldots,x_k \rbrace$ of vertices.
Therefore, for every $e \in \binom{[n]}{k}$, 
the number of $E \in E(\mathcal{H}_{n,i})$ containing $e$ is at most
\begin{equation}\label{copycount}
v_i !\binom{n-k}{v_i-k}.
\end{equation}
Thus every $A \in \mathcal{F}_{n,i}$ satisfies
$$
|A| \geq \frac{e(\mathcal{H}_{n,i}[A])}{v_i!\binom{n-k}{v_i-k}} \stackrel{(\ref{eHni})}{\geq} \frac{\eps'v_i!\binom{n}{v_i}}{v_i!\binom{n-k}{v_i-k}|\mathrm{Aut}(H_i)|}= \frac{\eps'}{\binom{v}{k}|\mathrm{Aut}(H_i)|}\binom{n}{k}
\geq \eps''\binom{n}{k},
$$
where, in the final inequality, we used the fact that $|\mathrm{Aut}(H_i)| \leq v_i!$.
Note that $\eps'' < \eps'$.
So $\mathcal{H}_{n,i}$ is $(\mathcal{F}_{n,i},\eps'')$-dense and $|A| \geq \eps''\binom{n}{k}$ for all $A \in \mathcal{F}_{n,i}$.

Certainly $p \geq n^{-1/m_k(H_j)}$ for all $j \in [r]$.
By the choice of $c$, we then have
$$
\Delta_{\ell}(\mathcal{H}_{n,i}) \leq c \cdot p^{\ell-1} \frac{e(\mathcal{H}_{n,i})}{\binom{n}{k}}
$$
for all $i \in [r]$ and $\ell \in [m_i]$.
We have shown that $\mathcal{H}_{n,i}$ and $\mathcal{F}_{n,i}$ satisfy the hypotheses of Proposition~\ref{asymmulti} for all $i \in [r]$.
\medskip

Then Proposition~\ref{asymmulti} implies that there exists a family $\mathcal{S}_r \subseteq \prod_{i \in [r]}\mathcal{P}(V(\mathcal{H}_{n,i})) = \mathcal{P}(\binom{[n]}{k})^r$ and  functions $f':\mathcal{S}_r \to \prod_{i \in [r]}\overline{\mathcal{F}_{n,i}}$ and $g:\mathcal{I}(\mathcal{H}_{n,1},\ldots,\mathcal{H}_{n,r}) \to \mathcal S_r$
such that the following conditions hold:
\begin{itemize}
\item[(a)] If $(S_1,\dots, S_r) \in \mathcal S_r$ then $\sum |S_i|\leq Dp \binom{n}{k}$;
\item[(b)] every $S \in \mathcal{S}_r$ satisfies $S \in \mathcal{I}(\mathcal{H}_{n,1},\ldots,\mathcal{H}_{n,r})$;
\item[(c)] for every $(I_1,\dots,I_r) \in \mathcal{I}(\mathcal{H}_{n,1},\ldots,\mathcal{H}_{n,r})$, we have  that $S \subseteq (I_1,\dots,I_r) \subseteq S\cup f'(S)$, where $S:=g( I_1,\dots, I_r)$.
\end{itemize}

Note that $(G_1,\ldots,G_r) \in \mathcal{I}(\mathcal{H}_{n,1},\ldots,\mathcal{H}_{n,r})$ if and only if $(G_1,\ldots,G_r) \in \mathcal{I}_r(n;H_1,\ldots,H_r)$ (where we recall the identification of graphs and edge sets).
For each $S \in \mathcal{S}_r$, define
$$
f(S) := S \cup f'(S).
$$
So $f : \mathcal{S}_r \rightarrow \mathcal{P}(\binom{[n]}{k})^r$. (Note that under the correspondence of graphs and edge sets we can view $\mathcal{P}(\binom{[n]}{k})^r=(\mathcal{G}_k(n))^r$.)
Thus (a)--(c) immediately imply that (i) and (iii) hold, and additionally for  any $(S_1,\ldots,S_r) \in \mathcal{S}_r$  we have
$$
\sum_{i \in [r]}e(S_i) \leq Dp\binom{n}{k} \leq Dn^{-1/m_k(H_1)} \cdot \frac{n^k}{k!} < Dn^{k-1/m_k(H_1)},
$$
yielding (ii).

Given any $S=(S_1,\ldots,S_r) \in \mathcal{S}_r$ write $f(S) =: (f(S_1),\ldots,f(S_r))$ and $f'(S) = :(f'(S_1),\ldots,f'(S_r))$.
Let $G := \bigcup_{i \in [r]}f(S_i)$; so $G$ is a $k$-graph with vertex set $[n]$.
To prove (iv)(a), we need to exhibit an $r$-colouring $\sigma$ of $G$ with the property that $\sigma^{-1}(i)$ contains less than $\eps\binom{n}{v_i}$ copies of $H_i$ for all $i \in [r]$.
Indeed, consider the $r$-colouring $\sigma$ of $G$ defined by setting $\sigma(e)=i$ when $i$ is the least integer such that $e \in f(S_i)$.
Then the subgraph of $G$ coloured $i$ is $\sigma^{-1}(i) \subseteq f(S_i) = S_i \cup f'(S_i)$.
Since $S_i$ is an independent set in $\mathcal{H}_{n,i}$, we have that $S_i$ is $H_i$-free.
Every copy of $H_i$ in $\sigma^{-1}(i)$ either contains at least one edge in $S_i$, or has every edge contained in $f'(S_i)$.
Note that  $m_k(H_1)\leq m$.
By~(\ref{copycount}), the number of copies of $H_i$ in $G$ containing at least one edge in $S_i$ is at most
\begin{align*}
e(S_i) \cdot v_i!\binom{n-k}{v_i-k} &\leq Dn^{k-1/m_k(H_1)} \cdot v_i! {(n-k)^{v_i-k}}
\leq {Dv_i!}\cdot n^{v_i-\frac{1}{m}} < \frac{\eps}{2}\binom{n}{v_i}.
\end{align*}
For each $i \in [r]$ we have that $f'(S_i) \in \overline{\mathcal{F}_{n,i}}$, and so $e(\mathcal{H}_{n,i}[f'(S_i)]) < \eps' e(\mathcal{H}_{n,i})$.
That is, the number of copies of $H_i$ in $f'(S_i)$ is less than
$$
\eps' \cdot \frac{v_i!}{|\mathrm{Aut}(H_i)|}\binom{n}{v_i} \leq \frac{\eps}{2}\binom{n}{v_i}.
$$
Thus, in total $f(S_i) = S_i \cup f'(S_i)$ contains at most $\eps\binom{n}{v_i}$ copies of $H_i$, so $G$ is $\eps$-weakly $(H_1,\ldots,H_r)$-Ramsey.
Since $\eps \leq \delta$, this immediately implies (iv)(a), and (iv)(b) follows from Theorem~\ref{epsclose}, our choice of parameters, and since $n$ is sufficiently large.
\end{proof}

As in Theorem~\ref{BMS}, we will call the elements $S \in \mathcal{S}_r$ \emph{fingerprints}, and each $\bigcup_{i \in [r]}f(S_i)$ with $(S_1,\ldots,S_r)\in \mathcal{S}_r$ is a \emph{container}.

\subsection{The number of hypergraphs which are not Ramsey}\label{sec54}

Our first application of Theorem~\ref{ramseycont} is an enumeration result for non-$(H_1,\ldots,H_r)$-Ramsey hypergraphs (Theorem~\ref{ramseycount}), which asymptotically determines the logarithm of $|\overline{\mathrm{Ram}}(n;H_1,\ldots,H_r)|$.

\medskip
\noindent
\emph{Proof of Theorem~\ref{ramseycount}.}
Let $0 < \delta < 1$ be arbitrary, and let $n \in \mathbb{N}$ be sufficiently large.
Clearly, $|\overline{\mathrm{Ram}}(n;H_1,\ldots,H_r)| \geq 2^{\ex^r(n;H_1,\ldots,H_r)}$ since no subhypergraph of an $n$-vertex non-$(H_1,\ldots,H_r)$-Ramsey $k$-graph with $\ex^r(n;H_1,\ldots,H_r)$ edges is $(H_1,\ldots,H_r)$-Ramsey.

For the upper bound,
suppose first that $\Delta_1(H_i)\geq 2$ for all $i \in [r]$.
Let $D>0$ be obtained from Theorem~\ref{ramseycont} applied to $H_1,\ldots,H_r$ with parameter $\delta$.
We obtain a collection $\mathcal{S}_r$ and a function $f$ as in Theorem~\ref{ramseycont}.
Consider any $G \in \overline{\mathrm{Ram}}(n;H_1,\ldots,H_r)$.
Note that there are pairwise edge-disjoint $k$-graphs $G_1,\ldots,G_r$ such that $\bigcup_{i\in[r]}G_i=G$ and $(G_1,\ldots,G_r) \in \mathcal{I}_r(n;H_1,\ldots,H_r)$.
So by Theorem~\ref{ramseycont}(i) this means there is some $S=(S_1,\ldots,S_r) \in \mathcal{S}_r$ so that $G \subseteq \bigcup_{i \in [r]}f(S_i)$.
Further, given any $S=(S_1,\ldots,S_r) \in \mathcal{S}_r$, we have
$$
e\left(\bigcup_{i \in [r]}f(S_i)\right) \leq \ex ^r(n;H_1,\ldots,H_r)+\delta \binom{n}{k}.
$$
Thus, each such $\bigcup_{i\in[r]}f(S_i)$ contains at most $2^{\ex ^r(n;H_1,\ldots,H_r)+\delta\binom{n}{k}}$ $k$-graphs in $\overline{\mathrm{Ram}}(n;H_1,\ldots,H_r)$.
Note that, by Theorem~\ref{ramseycont}(ii),
$$
|\mathcal{S}_r| \leq \left(\sum_{s=0}^{Dn^{k-1/m_k(H_1)}} \binom{\binom{n}{k}}{s}\right)^r < 2^{\delta\binom{n}{k}},
$$
where the last inequality holds since $n$ is sufficiently large.
Altogether, this implies
\begin{equation}\label{ubDelta}
|\overline{\mathrm{Ram}}(n;H_1,\ldots,H_r)| \leq 2^{\delta \binom{n}{k}} \times 2^{\ex ^r(n;H_1,\ldots,H_r)+\delta\binom{n}{k}} = 2^{\ex ^r(n;H_1,\ldots,H_r)+2\delta\binom{n}{k}}.
\end{equation}
Since the choice of $0 < \delta < 1$ was arbitrary, this proves the theorem in the case when $\Delta_1(H_i) \geq 2$ for all $i \in [r]$.

Suppose now that, say, $\Delta_1(H_1) = 1$.
Then $H_1$ is a  matching.\COMMENT{AT: I think it is better to implicitly assume $\Delta (H_1)\not =0$; since next sentence is false in this case.}
Certainly every non-$(H_2,\ldots,H_r)$-Ramsey $k$-graph is non-$(H_1,\ldots,H_r)$-Ramsey.
Let $H \in \overline{\mathrm{Ram}}(n;H_1,\ldots,H_r)$.
Then there exists an $r$-colouring $\sigma$ of $H$ such that $\sigma^{-1}(i)$ is $H_i$-free for all $i \in [r]$.
Thus $H$ is the union of pairwise edge-disjoint $k$-graphs $J \in \overline{\mathrm{Ram}}(n;H_2,\ldots,H_r)$ and $J' := \sigma^{-1}(1)$.
But $J'$ is $H_1$-free and hence does not contain a matching of size $\lfloor v(H_1)/2\rfloor =: h$.
A result of Erd\H{o}s~\cite{erdosmatch} (used here in a weaker form) implies that, for sufficiently large $n$,
$$
e(J') \leq (h-1)\binom{n-1}{k-1}.
$$
Thus, for large $n$,
\begin{align*}
|\overline{\mathrm{Ram}}(n;H_1,\ldots,H_r)| &\leq \sum_{J \in \overline{\mathrm{Ram}}(n;H_2,\ldots,H_r)}
\sum_{e(J')=0}^{(h-1)\binom{n-1}{k-1}} \binom{\binom{n}{k}}{e(J')} \\
&= |\overline{\mathrm{Ram}}(n;H_2,\ldots,H_r)| \sum_{e(J')=0}^{\frac{k(h-1)}{n}\binom{n}{k}}  
\binom{\binom{n}{k}}{e(J')}\\
&\leq |\overline{\mathrm{Ram}}(n;H_2,\ldots,H_r)| \cdot 2^{\delta\binom{n}{k}}.
\end{align*}
Iterating this argument, using~(\ref{ubDelta}) and the fact that $0 < \delta < 1$ was arbitrary, we obtain the required upper bound in the general case.
\hfill$\square$

\medskip
In fact Theorem~\ref{ramseycount} can be recovered in a different way, which, to the best of our knowledge,
has not been explicitly stated elsewhere.
Let $\mathcal{F}$ be a (possibly infinite) family of $k$-graphs, and let $\mathrm{Forb}(n;\mathcal{F})$ be the set of $n$-vertex $k$-graphs which contain no copy of any $F \in \mathcal{F}$ as a subhypergraph.
The following result of Nagle, R\"odl and Schacht~\cite{nrs} asymptotically determines the logarithm of $|\mathrm{Forb}(n;\mathcal{F})|$.
(This generalises the corresponding result of Erd\H{o}s, Frankl and R\"odl~\cite{efr} for graphs.)
Let
$$
\ex(n;\mathcal{F}) := \max\lbrace e(H):H \in \mathrm{Forb}(n;\mathcal{F})\rbrace. 
$$
(So when $\mathcal{F} = \lbrace F \rbrace$ contains a single $k$-graph, we have $\ex(n;\lbrace F \rbrace) = \ex(n;F)$.)

\COMMENT{AT: mention victor here?}

\begin{theorem}[Theorem 2.3, \cite{nrs}]\label{forbcount}
Let $k \geq 2$ be a positive integer and $\mathcal{F}$ be a (possibly infinite) family of $k$-graphs. Then, for all $n \in \mathbb{N}$,
$$
|\mathrm{Forb}(n;\mathcal{F})| = 2^{\ex(n;\mathcal{F})+o(n^k)}.
$$
\end{theorem}

Since $G \in \overline{\mathrm{Ram}}(n;H_1,\ldots,H_r)$ if and only if $G$ is an $n$-vertex $k$-graph without a copy of any $F \in \mathrm{Ram}(H_1,\ldots,H_r)$ as a subhypergraph, Theorem~\ref{forbcount} immediately implies Theorem~\ref{ramseycount}.

\subsection{The resilience of being \texorpdfstring{$(H_1,\ldots,H_r)$}{(H_1,...,H_r)}-Ramsey}\label{randomsec}

Recall that $G^{(k)}_{n,p}$ has vertex set $[n]$, where each edge lies in $\binom{[n]}{k}$ and appears with probability $p$, independently of all other edges.
In this section we apply Theorem~\ref{ramseycont} to prove Theorem~\ref{ramres}, which determines $\res(G^{(k)}_{n,p},(H_1,\ldots,H_r)\text{-Ramsey})$ for given fixed $k$-graphs $H_1,\ldots,H_r$.
Explicitly, $\res(G^{(k)}_{n,p},(H_1,\ldots,H_r)\text{-Ramsey})$ is the minimum integer $t$ such that one can remove $t$ edges from $G^{(k)}_{n,p}$ to obtain a $k$-graph $H$ which has an $(H_1,\ldots,H_r)$-free $r$-colouring.

Observe that Theorem~\ref{ramres} together with~(\ref{pibound}) immediately implies the following corollary.

\begin{corollary}[Random Ramsey for hypergraphs]\label{asymhyprandomramsey}
For all positive integers $r,k$ with $k \geq 2$ and $k$-graphs $H_1,\ldots,H_r$ with $m_k(H_1) \geq \ldots \geq m_k(H_r)$ and $\Delta_1(H_i) \geq 2$ for all $i \in [r]$, there exists $C > 0$ such that
$$
\lim_{n \rightarrow \infty} \mathbb{P}\left[ G^{(k)}_{n,p} \text{ is } (H_1,\ldots,H_r)\text{-Ramsey} \right] = 1\quad\text{if } p > Cn^{-1/m_k(H_1)}.
$$
\end{corollary}
In the case when $m_k(H_1)=m_k(H_2)$, Corollary~\ref{asymhyprandomramsey} generalises Theorem~\ref{asymmthm} since we do not require $H_1$ to be strictly $k$-balanced. Further, Corollary~\ref{asymhyprandomramsey} resolves (the $1$-statement part) of Conjecture~\ref{conjkreu} in the case when 
$m_2(H_1)=m_2(H_2)$.

\medskip
\noindent
\emph{Proof of Theorem~\ref{ramres}.}
Let $0 < \delta < 1$ be arbitrary, $r,k \in \mathbb{N}$ with $k \geq 2$, and let $H_1,\ldots,H_r$ be $k$-graphs as in the statement of the theorem.
Given $n\in \mathbb{N}$, if $p > n^{-1/m_k(H_1)}$, then $p > n^{-(k-1)}$ since $\Delta_1(H_1)\geq 2$.
Proposition~\ref{chernoff} implies that, w.h.p.,
\begin{equation}\label{concentration}
e(G^{(k)}_{n,p}) = \left(1\pm \frac{\delta}{4}\right)p\binom{n}{k}.
\end{equation}
For brevity, write $\pi := \pi(H_1,\ldots,H_r)$.
We will first prove the upper bound
$$
\lim_{n \rightarrow \infty} \mathbb{P}\left[\res(G^{(k)}_{n,p},(H_1,\ldots,H_r)\text{-Ramsey}) \leq (1-\pi + \delta)e(G^{(k)}_{n,p})\right] = 1 \quad\text{if}\quad p > n^{-1/m_k(H_1)}.
$$
For this, we must show that the probability of the event that there exists an $n$-vertex $k$-graph $G\subseteq G^{(k)}_{n,p} $ such that $e(G) \geq (\pi-\delta)e(G^{(k)}_{n,p})$ and $G \in \overline{\mathrm{Ram}}(n;H_1,\ldots,H_r)$, tends to one as $n$ tends to infinity.
This indeed follows:
Let $n$ be sufficiently large so that $\ex^r(n;H_1,\ldots,H_r) \geq (\pi-\delta/2)\binom{n}{k}$.
Let $G^*$ be an $n$-vertex non-$(H_1,\ldots,H_r)$-Ramsey $k$-graph with $e(G^*) = \ex^r(n;H_1,\ldots,H_r)$.
Then, by Proposition~\ref{chernoff}, w.h.p. we have $e(G^* \cap G^{(k)}_{n,p}) = (\pi \pm \delta)e(G^{(k)}_{n,p})$, and $G^*\cap G^{(k)}_{n,p} \in \overline{\mathrm{Ram}}(n;H_1,\ldots,H_r)$, as required.

 For the remainder of the proof, we will focus on the lower bound, namely that there exists $C>0$ such that whenever $p > Cn^{-1/m_k(H_1)}$,
\begin{equation}\label{aimeq}
\mathbb{P}\left[\res(G^{(k)}_{n,p},(H_1,\ldots,H_r)\text{-Ramsey}) \geq (1-\pi-\delta)e(G^{(k)}_{n,p})\right] \rightarrow 1\quad\text{as}\quad n \rightarrow \infty.
\end{equation}
Suppose $n$ is sufficiently large.
Apply Theorem~\ref{ramseycont} with parameters $r,k,\delta/16,(H_1,\ldots,H_r)$ to obtain $D > 0$ and for each $n \in \mathbb{N}$, a collection $\mathcal{S}_r$ and a function $f$ satisfying (i)--(iv).
Now choose $C$ such that $0 < 1/C \ll 1/D,\delta,1/k,1/r$.
Let $p \geq Cn^{-1/m_k(H_1)}$.

Since~(\ref{concentration}) holds with high probability, to prove~(\ref{aimeq}) holds it suffices to show that the probability $G^{(k)}_{n,p}$ contains a non-$(H_1,\ldots,H_r)$-Ramsey $k$-graph with at least $(\pi+\delta/2)p\binom{n}{k}$ edges tends to zero as $n$ tends to infinity.

Suppose that $G^{(k)}_{n,p}$ does contain a non-$(H_1,\ldots,H_r)$-Ramsey $k$-graph $I$ with at least $(\pi+\delta/2)p\binom{n}{k}$ edges.
Then there exist pairwise edge-disjoint $k$-graphs $I_1,\ldots,I_r$ such that $\bigcup_{i \in [r]}I_i=I$ and $(I_1,\ldots,I_r) \in \mathcal{I}_r(n;H_1,\ldots,H_r)$.
Further, there is some $S=(S_1,\ldots,S_r)\in \mathcal{S}_r$ such that $S \subseteq (I_1,\ldots,I_r) \subseteq f(S)$.
Thus, $G^{(k)}_{n,p}$ must contain (the edges of) $\bigcup_{i \in [r]}S_i$ as well as at least $(\pi+\delta/4)p\binom{n}{k}$ edges from $(\bigcup_{i \in [r]}f(S_i))\setminus(\bigcup_{i \in [r]}S_i)$.
(Note here we are using that $e(\bigcup_{i \in [r]}S_i) \leq \delta p\binom{n}{k}/4$, which holds by Theorem~\ref{ramseycont}(ii) and since $0 < 1/C \ll 1/D,1/k,\delta$.)
Writing $s := e(\bigcup_{i \in [r]}S_i)$, the probability $G^{(k)}_{n,p}$ contains $\bigcup_{i \in [r]}S_i$ is $p^s$.
Note that $e((\bigcup_{i \in [r]}f(S_i))\setminus(\bigcup_{i \in [r]}S_i)) \leq (\pi+\delta/8)\binom{n}{k}$ by Theorem~\ref{ramseycont}(iv)(b) and since $n$ is sufficiently large.
So by the first part of Proposition~\ref{chernoff}, the probability $G^{(k)}_{n,p}$ contains at least $(\pi+\delta/4)p\binom{n}{k}$ edges from $(\bigcup_{i \in [r]}f(S_i))\setminus(\bigcup_{i \in [r]}S_i)$ is at most $\exp(-\delta^2p\binom{n}{k}/256) \leq \exp(-\delta^2pn^k/256k^k)$.

Write $N := n^{k-1/m_k(H_1)}$ and $\gamma := \delta^2/256k^k$.
Given some integer $0 \leq s \leq DN$, there are at most $r^s\binom{\binom{n}{k}}{s}$ elements $(S_1,\ldots,S_r) \in \mathcal{S}_r$ such that $e(\cup_{i \in [r]}S_i)=s$.
Indeed, this follows since there are $r^s$ ways to partition a set of size $s$ into $r$ classes.
(Note we only need to consider $s \leq DN$ by Theorem~\ref{ramseycont}(ii).)
Thus, the probability that $G^{(k)}_{n,p}$ does contain a non-$(H_1,\ldots,H_r)$-Ramsey $k$-graph $I$ with at least $(\pi+\delta/2)p\binom{n}{k}$ edges is at most
\begin{align*}
\sum_{s=0}^{DN}r^s\binom{\binom{n}{k}}{s} \cdot p^s \cdot e^{-\gamma n^kp} &\leq (DN+1) (rp)^{DN} \binom{\binom{n}{k}}{DN} e^{-\gamma n^kp} \leq (DN+1)\left(\frac{re^{k+1}pn^k}{DNk^k}\right)^{DN} e^{-\gamma n^kp}\\
&\leq (DN+1)\left(\frac{re^{k+1}C}{Dk^k}\right)^{DN} e^{-\gamma CN} \leq e^{\gamma CN/2}e^{-\gamma CN} = e^{-\gamma CN/2},
\end{align*}
which tends to zero as $n$ tends to infinity.
This completes the proof.
\hfill$\square$

\section*{Acknowledgements}
The authors are grateful to Victor Falgas-Ravry for a helpful conversation on~\cite{templates}, \COMMENT{Mention Victor somewhere else?}
to Hong Liu for a helpful conversation on the graph Ramsey problems considered in this paper,
and to Yury Person for a helpful conversation on Theorem~\ref{ramres}.
The authors are also grateful to the reviewers for their helpful and careful reviews.


\begin{thebibliography}{9}
\bibitem{chvatal} V.~Chv\'atal, personal communication to the authors of~\cite{bes}.%
\COMMENT{is this how to cite such a thing????}
\end{thebibliography}
\end{document}